\documentclass[12pt]{article}
\usepackage{amssymb}
\usepackage{amsfonts}
\usepackage{amsmath}
\usepackage[svgnames]{xcolor}

\definecolor{sepia}{cmyk}{0, 0.83, 1, 0.70}

\newtheorem{theorem}{Theorem}[section]

\newtheorem{corollary}[theorem]{Corollary}

\newtheorem{definition}[theorem]{Definition}
\newtheorem{example}[theorem]{Example}

\newtheorem{lemma}[theorem]{Lemma}

\newtheorem{proposition}[theorem]{Proposition}
\newtheorem{fact}[theorem]{Fact}
\newtheorem{remark}[theorem]{Remark}

\newenvironment{proof}[1][\it{Proof}]{\noindent\textbf{#1.} }{\ \rule{0.0em}{0.0em}{\hfill}{$\square$}}
\input{tcilatex}
\begin{document}

\title{Regular Sequences of Quasi-Nonexpansive Operators and Their
Applications}
\author{Andrzej Cegielski\thanks{%
Faculty of Mathematics, Computer Science and Econometrics, University of
Zielona Gora, ul. Szafrana 4a, 65-516 Zielona Gora, Poland, e-mail:
a.cegielski@wmie.uz.zgora.pl}, Simeon Reich\thanks{%
Department of Mathematics, The Technion - Israel Institute of Technology,
3200003 Haifa, Israel, e-mail: sreich@technion.ac.il} \ and Rafa\l\ Zalas%
\thanks{%
Department of Mathematics, The Technion - Israel Institute of Technology,
3200003 Haifa, Israel, e-mail: rzalas@technion.ac.il} }
\maketitle

\begin{abstract}
In this paper we present a systematic study of regular sequences of
quasi-nonexpansive operators in Hilbert space. We are interested, in
particular, in weakly, boundedly and linearly regular sequences of
operators. We show that the type of the regularity is preserved under
relaxations, convex combinations and products of operators. Moreover, in
this connection, we show that weak, bounded and linear regularity lead to
weak, strong and linear convergence, respectively, of various iterative
methods. This applies, in particular, to block iterative and string
averaging projection methods, which, in principle, are based on the
above-mentioned algebraic operations applied to projections. Finally, we
show an application of regular sequences of operators to variational
inequality problems.


\end{abstract}

\section{Introduction}

\label{sec:intro} Let $\mathcal{H}$ be a real Hilbert space equipped with
inner product $\langle \cdot ,\cdot \rangle $ and induced norm $\left\Vert
\cdot \right\Vert $. We denote by $\limfunc{Fix}U:= \{x\in \mathcal{H}\mid
Ux=x\}$ the \textit{fixed point set} of an operator $U\colon\mathcal{H}%
\rightarrow \mathcal{H}$. We recall that for given closed and convex sets $%
C_i\subseteq\mathcal{H}$, $i=1,\ldots, m$, the \textit{convex feasibility
problem} (CFP) is to find a point $x$ in $C:=\bigcap_{i=1}^m C_i$. In this
paper we assume that the CFP is consistent, that is, $C\neq\emptyset$.

\textbf{Motivation.} Below we formulate a prototypical convergence theorem
for the methods of cyclic and simultaneous projections:

\begin{theorem}[\protect\cite{BB96}]
\label{intro:th:basic} Let $U:=\prod_{i=1}^m P_{C_i}$ or $U:=\frac 1 m
\sum_{i=1}^m P_{C_i}$ and for each $k=0,1,2\ldots,$ let $x^{k+1}:=Ux^k$,
where $x^0\in\mathcal{H}$. Then:

\begin{enumerate}
\item[(i)] $x^k$ converges weakly to some point $x^*\in C$.

\item[(ii)] If the family of sets $\{C_1,\ldots,C_m\}$ is boundedly regular,
then the convergence is in norm.

\item[(iii)] If the family of sets $\{C_1,\ldots,C_m\}$ is boundedly
linearly regular, then the convergence is linear.
\end{enumerate}
\end{theorem}

It is not difficult to see that both algorithmic operators $U$ in the above
theorem, due to the demi-closedness of $U-\limfunc{Id}$ at 0 \cite[Theorem 1]%
{Opi67}, for each $\{x^{k}\}_{k=0}^{\infty }\subseteq \mathcal{H}$ and $%
\{n_{k}\}_{k=0}^{\infty }\subseteq \{k\}_{k=0}^{\infty }$, satisfy 
\begin{equation}
\left. 
\begin{array}{l}
x^{n_{k}}\rightharpoonup y \\ 
Ux^{k}-x^{k}\rightarrow 0%
\end{array}%
\right\} \quad \Longrightarrow \quad y\in \limfunc{Fix}U,  \label{intro:WR}
\end{equation}%
where $\limfunc{Fix}U=C$. Moreover, note that in case (ii), by \cite[%
Theorems 4.10 and 4.11]{CZ14}, we have 
\begin{equation}
\lim_{k\rightarrow \infty }\Vert Ux^{k}-x^{k}\Vert =0\quad \Longrightarrow
\quad \lim_{k\rightarrow \infty }d(x^{k},\limfunc{Fix}U)=0,  \label{intro:BR}
\end{equation}%
which holds for any bounded sequence $\{x^{k}\}_{k=0}^{\infty }\subseteq 
\mathcal{H}$. Finally, in case (iii), we have observed, as will be shown
below, that for any bounded subset $S\subseteq \mathcal{H}$, there is $%
\delta >0$ such that for all $x\in S$, we have 
\begin{equation}
d(x,\limfunc{Fix}U)\leq \delta \Vert Ux-x\Vert \text{.}  \label{intro:LR}
\end{equation}%
It turns out that, in principle, conditions \eqref{intro:WR}, %
\eqref{intro:BR} and \eqref{intro:LR} are intrinsic abstract properties of $%
U $ which, when combined with the strong quasi-nonexpansivity, lead to weak,
strong and linear convergence; see, for example, \cite{BNP15} and \cite%
{KRZ17}. In this paper we refer to them as \textit{weak, bounded} and 
\textit{linear regularity} of the given operator $U$, respectively; see
Definition \ref{d-R}. Note that the iterative methods described in Theorem %
\ref{intro:th:basic} are static, that is, we iterate one fixed algorithmic
operator $U$. Nevertheless, in many cases, the iterative methods applied to
solving the CFPs are dynamic in the sense the algorithmic operators may
change from iteration to iteration. More precisely, one considers the
following general form of the iterative method: 
\begin{equation}
x^{0}\in \mathcal{H};\qquad x^{k+1}:=U_{k}x^{k},  \label{intro:genIter}
\end{equation}%
where for each $k=0,1,2,\ldots $, $U_{k}\colon \mathcal{H}\rightarrow 
\mathcal{H}$ is quasi-nonexpansive and satisfies $C\subseteq \limfunc{Fix}%
U_{k}$. The examples of \eqref{intro:genIter} with an extensive survey can
be found in \cite{Ceg12}; see also Example \ref{ex-BRk}. The study of
dynamic iterative methods necessitates a systematic investigation of the
abstract properties of the sequences of regular operators. The main
properties that we are interested in are related not only to convex
combination and products of regular operators, as in Theorem \ref%
{intro:th:basic}, but also to relaxation, that is, to operators of the form $%
\limfunc{Id}+\alpha (U-\limfunc{Id})$, where $\alpha \in (0,2)$. All three
of these algebraic operations are, in principle, the building bricks for
block-iterative \cite{AC89, Com96, Com97, BB96}, dynamic string averaging 
\cite{AR08, BRZ18, CZ13} and even more sophisticated algorithms, such as
modular string averaging \cite{RZ16}.

\textbf{Contribution.} The main contribution of this paper consists in
extending the notion of weakly, boundedly and linearly regular operators
described in \eqref{intro:WR}, \eqref{intro:BR} and \eqref{intro:LR} by
replacing one fixed operator $U$ with a sequence of operators $\{U_{k}\}$.
Within the framework of this extension, we provide a systematic study of
sequences of regular operators, where we establish their basic properties
and give some examples. The main result in this direction is that the the
convex combination and product operations, when applied to regular sequences
of operators, preserve the initial regularity under certain conditions; see
Theorems \ref{t-Rk1} and \ref{t-Rk2}. Although the preservation of weak \cite%
[Theorems 4.1 and 4.2]{Ceg15a} and bounded regularity \cite[Theorems 4.10
and 4.11]{CZ14} was known for one fixed operator, the preservation of linear
regularity, even in this simple case, seems to be new; see Corollaries \ref%
{c-Rk1b} and \ref{c-Rk2b}. Next, we extend Theorem \ref{intro:th:basic} by
showing that weak, bounded and linear regularity, when combined with
appropriate regularity of sets and strong quasi-nonexpasivity, lead to weak,
strong and linear convergence of the method \eqref{intro:genIter}; see
Theorems \ref{th:main} and \ref{th:main2}. Moreover, following recent work
in the field of variational inequalities \cite{Ceg15, CZ13, CZ14, GRZ15,
GRZ17}, we provide an application of regular sequences of operators in this
direction as well; see Theorem \ref{t-uku}.

\textbf{Historical overview.} The regularity properties described in %
\eqref{intro:WR}, \eqref{intro:BR} and \eqref{intro:LR} have reappeared in
the literature under various names, as we now recall.

Clearly, a weakly regular operator $U$ is an operator for which $U-\limfunc{%
Id}$ is demi-closed at 0. This type of the demi-closedness condition goes
back to the papers by Browder and Petryshyn \cite{BP66} and by Opial \cite%
{Opi67}. The term \textit{weakly regular operator} was introduced in \cite[%
Def. 12]{KRZ17}. The concept of weak regularity has recently been extended
to the \textit{fixed point closed mappings} in \cite[Lemma 2.1]{BCW14}),
where the weak convergence was replaced by the strong one. Weakly regular
sequences of operators appeared already in \cite[Sec. 2]{AK14}, where they
were called sequences satisfying \textit{condition (Z)} and applied to a
viscosity approximation process for solving variational inequalities. Weakly
regular sequences of operators were also studied in \cite{Ceg15}, where they
were introduced through sequences satisfying a \textit{demi-closedness
principle}, again, with applications to variational inequalities. Some
properties of weakly regular sequences of operators can be found in \cite%
{RZ16}.

A prototypical version of regular operator can be found in \cite[Theorem 1.2]%
{PW73} by Petryshyn and Williamson, where it was assumed, in addition, that
the operator was continuous. As far we know, the definition of boundedly
regular operators as well as their properties were first proposed by
Cegielski and Zalas in \cite[Definition 16]{CZ13} under the name \textit{%
approximately shrinking}, because of their relation to the \textit{%
quasi-shrinking} operators defined in \cite[Section 3]{YO04}. The term
``boundedly regular operator'' was proposed in \cite[Definition 7.1]{BNP15}.
Because of the relationship of boundedly/linearly regular operators to
boundedly/linearly regular families of sets (see Remark \ref{r-BRsets}), in
this paper we have replaced the term ``approximately shrinking'' by
``boundedly regular''. Many properties of these operators under the name
``approximately shrinking'' were presented in \cite{CZ14} with some
extensions in \cite{Zal14}, \cite{RZ16} and \cite{Ceg16}, and with more
applications in \cite{Ceg15} and \cite{CM16}. It is worth mentioning that
regular operators were applied even in Hadamard spaces to solving common
fixed point problems \cite{RS17}.

The phrase \textit{boundedly linearly regular} in connection to operators
was proposed by Bauschke, Noll and Phan, who applied them to establish a
linear rate of convergence for some block iterative fixed point algorithms 
\cite[Theorem 6.1]{BNP15}. To the best of our knowledge, the concept of this
type of operator goes back to Outlaw \cite[Theorem 2]{Out69}, and Petryshyn
and Williamson \cite[Corollary 2.2]{PW73}. A closely related condition
called a \textit{linearly focusing} algorithm, was studied by Bauschke and
Borwein \cite[Definition 4.8]{BB96}. The concept of a focusing algorithm
goes back to Fl\aa m and Zowe \cite[Section 2]{FZ90}, and can also be found
in \cite[Definition 1.2]{Com97} by Combettes. Linearly regular operators
appeared in \cite[Definition 3.3]{CZ14} by Cegielski and Zalas as \textit{%
linearly shrinking} ones.

We would like to mention that in the literature one can find concepts
similar to our concepts of regularity of operators; see, for example, 
\textit{H\"{o}lder regular operators }in \cite[Definition 2.4]{BLT17} or 
\textit{modulus of regularity} in \cite[Definition 3.1]{KLN17}.

\textbf{Organization of the paper.} In Section \ref{s-prel} we introduce the
reader to our notation and to basic facts regarding quasi-nonexpansive
operators, Fej\'{e}r monotone sequences and regular families of sets. In
Section \ref{s-RO} we formulate the definition of regular operators and give
several examples. In Section \ref{s-RS} we extend this definition to
sequences of operators and show their basic properties. The main properties
related to sequences, but not limited to them, are presented in Section \ref%
{s-4}. Applications to convex feasibility problems and variational
inequalities are shown in Section \ref{s-app}.

\section{Preliminaries\label{s-prel}}

\textbf{Notation. } Sequences of elements of $\mathcal{H}$ will be denoted
by $x^{k},y^{k},z^{k}$, etc. Sequences of real parameters will be usually
denoted by $\alpha _{k},\lambda _{k},\omega _{k}$ or by $\rho
_{i}^{k},\omega _{i}^{k}$, etc. Sequences of operators will be denoted by $%
\{T_{k}\}_{k=0}^{\infty },\{U_{k}\}_{k=0}^{\infty }$ or by $%
\{U_{i}^{k}\}_{k=0}^{\infty }$ etc. In order to distinguish $\rho _{i}^{k}$
and $U_{i}^{k}$ from the $k$-th power of $\rho _{i}$ and $U_{i}$, the latter
will be denoted by $(\rho _{i})^{k}$ and $(U_{i})^{k}$, respectively. We
denote the \textit{identity} operator by $\limfunc{Id}$. For a family of
operators $U_{i}:\mathcal{H}\rightarrow \mathcal{H}$, $i\in I:=\{1,2,...,m\}$%
, and an ordered set $K:=(i_{1},i_{2},...,i_{s})$, we denote $\tprod_{i\in
K}U_{i}:=U_{i_{s}}U_{i_{s-1}}...U_{i_{1}}$. For an operator $T$ and for $%
\lambda \geq 0$ we define $T_{\lambda}:=\limfunc{Id}+\lambda (T-\limfunc{Id}%
) $ and call it a $\lambda $-\textit{relaxation} of $T$, while $\lambda $ is
called the \textit{relaxation parameter}. For $\alpha \in 
\mathbb{R}
$, denote $\alpha _{+}:=\max \{0,\alpha \}$. Similarly, for a function $f:%
\mathcal{H}\rightarrow 
\mathbb{R}
$, denote $f_{+}:=\max \{0,f\}$, that is, $f_{+}(x)=[f(x)]_{+}$, $x\in 
\mathcal{H}$. For a fixed $x\in \mathcal{H}$, denote $\limfunc{Argmin}_{i\in
I}f_{i}(x)=\{j\in I \mid f_{j}(x)\leq f_{i}(x)$ for all $i\in I\}$. $\square$

Let $C\subseteq \mathcal{H}$ be nonempty, closed and convex. It is well
known that for any $x\in \mathcal{H}$, there is a unique point $y\in C$ such
that $\Vert x-y\Vert \leq \Vert x-z\Vert $ for all $z\in C$. This point is
called the \textit{metric projection} of $x$ onto $C$ and is denoted by $%
P_{C}x$. The operator $P_{C}:\mathcal{H}\rightarrow \mathcal{H}$ is
nonexpansive and $\limfunc{Fix}P_{C}=C$. Moreover, $P_{C}x$ is characterized
by: $y\in C$ and $\langle z-y,x-y\rangle \leq 0$ for all $z\in C$.

Let $f:\mathcal{H}\rightarrow 
\mathbb{R}
$ be a convex continuous function. Then for any $x\in \mathcal{H}$, there
exists a point $g_{f}(x)\in \mathcal{H}$ satisfying $\langle
g_{f}(x),y-x\rangle \leq f(y)-f(x)$ for all $y\in \mathcal{H}$. This point
is called a \textit{subgradient} of $f$ at $x$. Suppose that $%
S(f,0):=\{x:f(x)\leq 0\}\neq \emptyset $. For each $x\in \mathcal{H}$, we
fix a subgradient $g_{f}(x)\in \mathcal{H}$ and define an operator $P_{f}:%
\mathcal{H}\rightarrow \mathcal{H}$ by 
\begin{equation}
P_{f}(x):=\left\{ 
\begin{array}{ll}
x-\frac{f(x)}{\Vert g_{f}(x)\Vert ^{2}}g_{f}(x)\text{,} & \text{if }f(x)>0%
\text{,} \\ 
x\text{,} & \text{otherwise.}%
\end{array}%
\right.
\end{equation}%
In order to simplify the notation we also write $P_{f}(x)=x-\frac{f_{+}(x)}{%
\Vert g_{f}(x)\Vert ^{2}}g_{f}(x)$ for short. The operator $P_{f}$ is called
a \textit{subgradient projection}. Clearly, $\limfunc{Fix}P_{f}=S(f,0)$.

Now we recall an inequality related to convex functions in $\mathbb{R}^{n}$.

\begin{lemma}[{\protect\cite[Lemma 3.3]{Fuk84}}]
\label{l-Fuk} Let $f\colon \mathbb{R}^{n}\rightarrow \mathbb{R}$ be convex
and assume that the Slater condition is satisfied, that is, $f(z)<0$ for
some $z\in \mathbb{R}^{n}$. Then for each compact subset $K$ of $\mathbb{R}%
^{n}$, there is $\delta >0$ such that the inequality 
\begin{equation}
\delta d(x,S(f,0))\leq f_{+}(x)  \label{e-Fuk}
\end{equation}%
holds for every $x\in K$.
\end{lemma}

\subsection{Strongly quasi-nonexpansive operators}

In this subsection we recall the notion of a strongly quasi-nonexpansive
operator as well as several properties of these operators.

\begin{definition}
\rm\ %
We say that $T$ is $\rho $-\textit{strongly quasi-nonexpansive }($\rho $%
-SQNE), where $\rho \geq 0$, if $\limfunc{Fix}T\neq \emptyset $ and 
\begin{equation}
\left\Vert Tu-z\right\Vert ^{2}\leq \left\Vert u-z\right\Vert ^{2}-\rho
\left\Vert Tu-u\right\Vert ^{2}  \label{e-SQNE}
\end{equation}%
for all $u\in \mathcal{H}$ and all $z\in \limfunc{Fix}T$. If $\rho =0$ in (%
\ref{e-SQNE}), then $T$ is called \textit{quasi-nonexpansive }(QNE). If $%
\rho >0$ in (\ref{e-SQNE}), then we say that $T$ is \textit{strongly
quasi-nonexpansive }(SQNE).
\end{definition}

Clearly, a nonexpansive operator having a fixed point is QNE. We say that $T$
is a \textit{cutter} if $\limfunc{Fix}T\neq \emptyset $ and $\langle
x-Tx,z-Tx\rangle \leq 0$ for all $x\in \mathcal{H}$ and for all $z\in 
\limfunc{Fix}T$. Now we recall well-known facts which we employ in the
sequel.

\begin{fact}
\label{f-3}If $T$ is QNE, then $\limfunc{Fix}T$ is closed and convex.
\end{fact}

\begin{fact}
\label{f-0}If $T$ is a cutter, then $\left\Vert Tx-x\right\Vert \leq
\left\Vert P_{\limfunc{Fix}T}x-x\right\Vert $ for all $x\in \mathcal{H}$.
\end{fact}

\begin{fact}
\label{f-2}The following conditions are equivalent:

\begin{enumerate}
\item[$\mathrm{(i)}$] $T$ is a cutter;

\item[$\mathrm{(ii)}$] $\langle Tx-x,z-x\rangle \geq \left\Vert
Tx-x\right\Vert ^{2}$ for all $x\in \mathcal{H}$ and for all $z\in \limfunc{%
Fix}T$;

\item[$\mathrm{(iii)}$] $T$ is $1$-SQNE;

\item[$\mathrm{(iv)}$] $T_{\lambda }$ is $(2-\lambda )/\lambda $-SQNE, where 
$\lambda \in (0,2]$.
\end{enumerate}
\end{fact}

For proofs of Facts \ref{f-3}--\ref{f-2}, see, for example, \cite[Section
2.1.3]{Ceg12}.

\begin{corollary}
\label{c-1}The following conditions are equivalent:

\begin{enumerate}
\item[$\mathrm{(i)}$] $T$ is QNE;

\item[$\mathrm{(ii)}$] $T_{\lambda }$ is $(1-\lambda )/\lambda $-SQNE, where 
$\lambda \in (0,1]$;

\item[$\mathrm{(iii)}$] $T_{1/2}$ is a cutter.
\end{enumerate}
\end{corollary}

The most important examples of cutter operators are the metric projection $%
P_{C}$ onto a nonempty, closed and convex subset $C\subseteq \mathcal{H}$
(see, e.g., \cite[Sections 1.2 and 2.2]{Ceg12}) and a subgradient projection 
$P_{f}$ related to a continuous convex function $f:\mathcal{H}\rightarrow 
\mathbb{R}
$ with $S(f,0):=\{x\in \mathcal{H}:f(x)\leq 0\}\neq \emptyset $ (see, for
instance, \cite[Section 4.2]{Ceg12}).

The following two facts play an important role in the sequel.

\begin{fact}
\label{f-5} Let $U_{i}:\mathcal{H}\rightarrow \mathcal{H}$ be $\rho _{i}$%
-SQNE, $\rho _{i}\geq 0$, $i\in I:=\{1,2,...,m\}$, with $\bigcap_{i\in I}%
\func{Fix}U_{i}\neq \emptyset $, $U:=\sum_{i\in I}\omega _{i}U_{i}$, where $%
\omega _{i}\geq 0,i\in I$, and $\sum_{\in I}\omega _{i}=1$.

\begin{enumerate}
\item[$\mathrm{(i)}$] If $\omega _{i},\rho _{i}>0$ for all $i\in I$, then $%
\limfunc{Fix}U=\bigcap_{i=1}^{m}\limfunc{Fix}U_{i}$ and $U$ is $\rho $-SQNE
with $\rho =\min_{i\in I}\rho _{i}$;

\item[$\mathrm{(ii)}$] For any $x\in \mathcal{H}$ and $z\in \bigcap_{i\in I}%
\func{Fix}U_{i}$ we have 
\begin{equation}
\Vert Ux-z\Vert ^{2}\leq \Vert x-z\Vert ^{2}-\sum_{i=1}^{m}\omega _{i}\rho
_{i}\Vert U_{i}x-x\Vert ^{2}\text{.}  \label{e-SQNE1}
\end{equation}

\item[$\mathrm{(iii)}$] For any $z\in \bigcap_{i\in I}\limfunc{Fix}U_i$, $%
x\in\mathcal{H}$ and positive $R\geq \Vert x-z\Vert$, we have 
\begin{equation}
\frac{1}{2R}\sum_{i=1}^{m}\omega _{i}\rho _{i}\Vert U_{i}x-x\Vert ^{2}\leq
\Vert Ux-x\Vert \text{.}  \label{e-SQNE2}
\end{equation}
\end{enumerate}
\end{fact}

\begin{proof}
For (i), see \cite[Theorems 2.1.26(i) and 2.1.50]{Ceg12}. Parts (ii) and
(iii) were proved in \cite[Proposition 4.5]{CZ14} in the case where $\rho >0$%
, but it follows from the proof that the statement is also true if $\rho
\geq 0$.
\end{proof}

\begin{fact}
\label{f-6} Let $U_{i}:\mathcal{H}\rightarrow \mathcal{H}$ be $\rho _{i}$%
-SQNE, $\rho _{i}\geq 0$, $i\in I:=\{1,2,...,m\}$, with $\bigcap_{i\in I}%
\func{Fix}U_{i}\neq \emptyset $, and let $U:=U_{m}U_{m-1}...U_{1}$.

\begin{enumerate}
\item[$\mathrm{(i)}$] If $\rho =\min_{i\in I}\rho _{i}>0$, then $\limfunc{Fix%
}U=\bigcap_{i=1}^{m}\limfunc{Fix}U_{i}$ and $U:=U_{m}U_{m-1}...U_{1}$ is $%
\rho /m$-SQNE;

\item[$\mathrm{(ii)}$] For any $x\in \mathcal{H}$ and $z\in \bigcap_{i\in I}%
\func{Fix}U_{i}$ we have 
\begin{equation}
\Vert Ux-z\Vert ^{2}\leq \Vert x-z\Vert ^{2}-\sum_{i=1}^{m}\rho _{i}\Vert
Q_{i}x-Q_{i-1}x\Vert ^{2}\text{,}  \label{e-SQNE3}
\end{equation}%
where $Q_{i}:=U_{i}U_{i-1}\ldots U_{1}$, $i\in I$, $Q_{0}:=\func{Id}$.

\item[$\mathrm{(iii)}$] For any $z\in \bigcap_{i\in I}\limfunc{Fix}U_i$, $%
x\in\mathcal{H}$ and positive $R\geq \Vert x-z\Vert$, we have 
\begin{equation}
\frac{1}{2R}\sum_{i=1}^{m}\rho _{i}\Vert Q_{i}x-Q_{i-1}x\Vert ^{2}\leq \Vert
Ux-x\Vert \text{.}  \label{e-SQNE4}
\end{equation}
\end{enumerate}
\end{fact}

\begin{proof}
For (i), see \cite[Theorems 2.1.26(ii) and 2.1.48(ii)]{Ceg12}. Parts (ii)
and (iii) were proved in \cite[Proposition 4.6]{CZ14} in the case where $%
\rho >0$, but it follows from the proof that the statement is also true if $%
\rho \geq 0 $.
\end{proof}

\subsection{Fej\'{e}r monotone sequences}

\begin{definition}
We say that a sequence $\{x^{k}\}_{k=0}^{\infty }$ is Fej\'{e}r monotone
with respect to a subset $C\subseteq \mathcal{H}$ if $\Vert x^{k+1}-z\Vert
\leq \Vert x^{k}-z\Vert $ for all $z\in C$ and $k=0,1,2,\ldots$.
\end{definition}

\begin{lemma}
\label{f-6a}Let $T_{k}$ be $\rho _{k}$-SQNE, $k=0,1,2,\ldots$, with $\rho
:=\inf_{k}\rho _{k}\geq 0$ and $F:=\bigcap_{k=0}^{\infty }\limfunc{Fix}%
T_{k}\neq \emptyset $, and let a sequence $\{x^{k}\}_{k=0}^{\infty }$ be
generated by $x^{k+1}=T_{k}x^{k}$, where $x^{0}\in \mathcal{H}$ is arbitrary.

\begin{enumerate}
\item[$\mathrm{(i)}$] The sequence $\{x^{k}\}_{k=0}^{\infty }$ is Fej\'{e}r
monotone with respect to $F$.

\item[$\mathrm{(ii)}$] If $\rho >0$, then $\lim_{k}\Vert
T_{k}x^{k}-x^{k}\Vert =0$.
\end{enumerate}
\end{lemma}

\begin{proof}
Part (i) follows directly from the definition of a QNE operator, while part
(ii) follows from (i) and from the definition of an SQNE operator.
\end{proof}

\begin{fact}
\label{f-7}If a sequence $\{x^{k}\}_{k=0}^{\infty }\subseteq \mathcal{H}$ is
Fej\'{e}r monotone with respect to a nonempty subset $C\subseteq \mathcal{H}$%
, then

\begin{enumerate}
\item[$\mathrm{(i)}$] $x^{k}$ converges weakly to a point $z\in C$ if and
only if all its weak cluster points belong to $C$;

\item[$\mathrm{(ii)}$] $x^{k}$ converges strongly to a point $z\in C$ if and
only if $\lim_{k}d(x^{k},C)=0$;

\item[$\mathrm{(iii)}$] if there is a constant $q\in (0,1)$ such that $%
d(x^{k+1},C)\leq qd(x^{k},C)$ holds for all $k=0,1,2,\ldots$, then $%
\{x^{k}\}_{k=0}^{\infty }$ converges linearly to a point $z\in C$ and 
\begin{equation}
\Vert x^{k}-z\Vert \leq 2d(x^{0},C)q^{k}\text{.}
\end{equation}
\end{enumerate}
\end{fact}

\begin{proof}
See \cite[Theorem 2.16(ii), (v) and (vi)]{BB96}.
\end{proof}

\begin{lemma}
\label{th:Fejer2} Let $\{x^{k}\}_{k=0}^{\infty }$ be Fej\'{e}r monotone with
respect to $C$ and let $s\in \mathbb{N}$.

\begin{enumerate}
\item[$\mathrm{(i)}$] If $x^{ks}\rightharpoonup z$ for some $z\in C$ and $%
\lim_{k}\Vert x^{k+1}-x^{k}\Vert =0$, then $x^{k}\rightharpoonup z$.

\item[$\mathrm{(ii)}$] If $x^{ks}\rightarrow z$ for some $z\in C$, then $%
x^{k}\rightarrow z$.

\item[$\mathrm{(iii)}$] If there are $c>0$, $q\in (0,1)$ and $z\in C$ such
that $\Vert x^{ks}-z\Vert \leq cq^{k}$ for each $k=0,1,2,\ldots$, then 
\begin{equation}
\Vert x^{k}-z\Vert \leq \frac{c}{(\sqrt[\scriptstyle{s}]{q})^{s-1}}\left( 
\sqrt[\scriptstyle{s}]{q}\right) ^{k}.
\end{equation}
\end{enumerate}
\end{lemma}

\begin{proof}
Suppose that the assumptions of (i) are satisfied. Let $n=n_{k}=\lfloor 
\frac{k}{s}\rfloor :=\max \{m\mid ms\leq k\}$ and $p=k-ns$. Clearly, $%
n\rightarrow \infty $ if and only if $k\rightarrow \infty $. By the
assumption, we have 
\begin{equation}
0\leq \lim_{k}\Vert x^{k}-x^{ns}\Vert =\lim_{n}\Vert x^{k}-x^{ns}\Vert \leq
\lim_{n}\sum_{l=ns}^{ns+p-1}\Vert x^{l+1}-x^{l}\Vert =0\text{.}
\end{equation}%
This yields $\lim_{k}\Vert x^{k}-x^{ns}\Vert =0$ and $%
x^{k}=x^{ns}+(x^{k}-x^{ns})\rightharpoonup z$ as $k\rightarrow \infty $.
Note that (i) is true even without the Fej\'{e}r monotonicity of $%
\{x^{k}\}_{k=0}^{\infty }$. Part (ii) follows from Fact \ref{f-7}(ii). For a
proof of (iii), see \cite[Prop. 1.6]{BB96}.
\end{proof}

\subsection{Regular families of sets}

Below we recall the notion of regularity of a finite family of sets as well
as several properties of regular families.

\begin{definition}
\rm\ %
(\cite[Def. 5.1]{BB96}, \cite[Def. 5.7]{BNP15}) Let $S\subseteq \mathcal{H}$
be nonempty and $\mathcal{C}$ be a family of closed convex subsets $%
C_{i}\subseteq \mathcal{H}$, $i\in I:=\{1,2,...,m\}$, with $C:=\bigcap_{i\in
I}C_{i}\neq \emptyset $. We say that $\mathcal{C}$ is:

\begin{enumerate}
\item[(a)] \textit{regular }over $S$ if for any sequence $%
\{x^{k}\}_{k=0}^{\infty }\subseteq S$, we have 
\begin{equation}
\lim_{k}\max_{i\in I}d(x^{k},C_{i})=0\Longrightarrow \lim_{k}d(x^{k},C)=0%
\text{;}
\end{equation}

\item[(b)] \textit{linearly regular }over $S$ if there is a constant $\kappa
>0$ such that for every $x\in S$, we have 
\begin{equation}
d(x,C)\leq \kappa \max_{i\in I}d(x,C_{i})\text{.}
\end{equation}%
We call the constant $\kappa $ a \textit{modulus} of the linear regularity
of $\mathcal{C}$ over $S$.
\end{enumerate}

If any of the above regularity conditions holds for $S=\mathcal{H}$, then we
omit the phrase \textquotedblleft over $S$\textquotedblright . \noindent If
any of the above regularity conditions holds for every bounded subset $%
S\subseteq \mathcal{H}$, then we precede the corresponding term with the
adverb \textit{boundedly} while omitting the phrase \textquotedblleft over $%
S $\textquotedblright .
\end{definition}

The theorem below gives a small collection of sufficient conditions for a
family $\mathcal{C}$ to be (boundedly, linearly) regular.

\begin{theorem}[\protect\cite{BB96, BNP15}]
\label{t-BRsets} Let $C_{i}\subseteq \mathcal{H}$, $i\in I:=\{1,\ldots ,m\}$%
, be closed convex with $C:=\bigcap_{i\in I}C_{i}\neq \emptyset $ and let $%
\mathcal{C}:=\{C_{i}\mid i\in I\}$.

\begin{enumerate}
\item[$\mathrm{(i)}$] If $\dim \mathcal{H}<\infty $, then $\mathcal{C}$ is
boundedly regular;

\item[$\mathrm{(ii)}$] If $C_{j}\cap \limfunc{int}(\bigcap_{i\in I\setminus
\{j\}}C_{i})\neq \emptyset $, then $\mathcal{C}$ is boundedly linearly
regular;

\item[$\mathrm{(iii)}$] If all $C_{i}$, $i\in I$, are half-spaces, then $%
\mathcal{C}$ is linearly regular;

\item[$\mathrm{(iv)}$] If $\dim \mathcal{H}<\infty $, $C_{j}$ is a
half-space, $j\in J\subseteq I$, and $\bigcap_{j\in J}C_{j}\cap
\bigcap_{i\in I\setminus J}\limfunc{ri}C_{i}\neq \emptyset $, then $\mathcal{%
C}$ is boundedly linearly regular.
\end{enumerate}
\end{theorem}

More sufficient conditions can be found, for example, in \cite[Fact 5.8]%
{BNP15}. Note that the bounded linear regularity of a family $\{C_{i}\mid
i\in I\}$ has no inheritance property even if each $C_{i}$, $i\in I$, is a
closed linear subspace \cite{RZ14}.

\section{Regular operators\label{s-RO}}

\begin{definition}
\label{d-R}%
\rm\ %
Let $S\subseteq \mathcal{H}$ be nonempty, and $C\subseteq \mathcal{H}$ be
nonempty, closed and convex. We say that a quasi-nonexpansive operator $%
U\colon\mathcal{H}\rightarrow\mathcal{H}$ is:

\begin{enumerate}
\item[(a)] \textit{weakly }$C$-\textit{regular} over $S$ if for any sequence 
$\{x^{k}\}_{k=0}^{\infty }\subseteq S$ and $\{n_k\}_{k=0}^\infty\subseteq
\{k\}_{k=0}^\infty$, we have 
\begin{equation}
\left . 
\begin{array}{l}
x^{n_k}\rightharpoonup y \\ 
\|U x^k-x^k\|\rightarrow 0%
\end{array}
\right\}\quad\Longrightarrow\quad y\in C;
\end{equation}

\item[(b)] $C$-\textit{regular} over $S$ if for any sequence $%
\{x^{k}\}_{k=0}^{\infty }\subseteq S$, we have 
\begin{equation}
\lim_{k\rightarrow \infty }\Vert Ux^{k}-x^{k}\Vert =0\quad \Longrightarrow
\quad \lim_{k\rightarrow \infty }d(x^{k},C)=0\text{;}
\end{equation}

\item[(c)] \textit{linearly }$C$-\textit{regular} over $S$ if there is $%
\delta >0$ such that for all $x\in S$, we have 
\begin{equation}
d(x,C)\leq \delta \Vert Ux-x\Vert \text{.}
\end{equation}%
The constant $\delta $ is called a \textit{modulus} of the linear $C$%
-regularity\textit{\ of }$U$ over $S$.
\end{enumerate}

If any of the above regularity conditions holds for $S=\mathcal{H}$, then we
omit the phrase ``over $S$''. If any of the above regularity conditions
holds for every bounded subset $S\subseteq \mathcal{H}$, then we precede the
corresponding term with the adverb ``boundedly'' while omitting the phrase
``over $S$'' (we allow $\delta $ to depend on $S$ in (c)). We say that $U$
is \textit{(boundedly) weakly regular}, \textit{regular} or \textit{linearly
regular} (over $S$) if $C=\limfunc{Fix}U$ in (a), (b) or (c), respectively.
\end{definition}

The most common setting of the above definition, in which we are interested,
is where 
\begin{equation}
C=\limfunc{Fix} U, \quad \limfunc{Fix} U\cap S \neq \emptyset \quad \text{and%
} \quad S \text{ is bounded}.
\end{equation}

\begin{remark}[Weak regularity]
\label{r-WR}%
\rm\ %
Note that if the operator $U$ is weakly regular, then this means that $U-%
\limfunc{Id}$ is demi-closed at 0. Observe that:

\begin{enumerate}
\item[(i)] $U$ is weakly $C$-regular over $S$ if and only if for any
sequence $\{x^{k}\}_{k=0}^{\infty }\subseteq S$, we have 
\begin{equation}
\left . 
\begin{array}{l}
x^k\rightharpoonup y \\ 
\|U x^k-x^k\|\rightarrow 0%
\end{array}
\right\}\quad\Longrightarrow\quad y\in C.
\end{equation}
This type of equivalence is no longer true for a $C$-weakly regular sequence
of operators, as we show in the next section; see Remark \ref{r-WR2}.

\item[(ii)] $U$ is boundedly weakly $C$-regular if and only if $U$ is weakly 
$C$-regular. This follows from (i) and the fact that any weakly convergent
sequence $\{x^k\}_{k=0}^\infty$ must be bounded. Therefore there is no need
to distinguish between boundedly weakly ($C$-)regular and weakly ($C$%
-)regular operators.
\end{enumerate}
\end{remark}

\begin{remark}[Regular operators and regular sets]
\label{r-BRsets}%
\rm\ %
The notion of regular operators is closely related to the notion of a
regular family of subsets. Indeed, for a family $\mathcal{C}$ of closed
convex subsets $C_{i}\subseteq \mathcal{H}$, $i\in I:=\{1,2,...,m\}$, having
a common point, denote by $P$ the metric projection onto the furthest subset 
$C_{i}$, that is, for any $x\in \mathcal{H}$ and for some $i(x)\in \limfunc{%
Argmax}_{i\in I}d(x,C_{i})$, $P(x)=P_{C_{i(x)}}x$. Note that, in general, $P$
is not uniquely defined, because, in general, $i(x)$ is not uniquely
defined. Therefore we suppose that for any $x\in \mathcal{H}$, the index $%
i(x)\in \limfunc{Argmax}_{i\in I}d(x,C_{i})$ is fixed, for example, $%
i(x)=\min \{i\in I\mid i\in \limfunc{Argmax}_{i\in I}d(x,C_{i})\}$. It is
easily seen that the operator $P$ is (linearly) regular over $S$ (with
modulus $\delta $) if and only if the family $\mathcal{C}$ is (linearly)
regular over $S$ (with modulus $\delta $).
\end{remark}

Clearly, the metric projection $P_{C}$ onto a nonempty closed convex subset $%
C\subseteq \mathcal{H}$ is linearly regular with a modulus $\delta =1$.
Below we give a few examples of weakly (boundedly, boundedly linearly)
regular operators.

\begin{example}
\label{ex-NE-BR}%
\rm\ %
A nonexpansive operator $U\colon \mathcal{H}\rightarrow \mathcal{H}$ with a
fixed point is weakly regular. This follows from the fact that a
nonexpansive operator satisfies the demi-closedness principle \cite[Lemma 2]%
{Opi67}. If $\mathcal{H}=\mathbb{R}^{n}$, then, by \cite[Proposition 4.1]%
{CZ14}, $U$ is boundedly regular. This, in principle, follows from the fact
that in $\mathbb{R}^{n}$ the weak convergence is equivalent to the strong
one. In this paper we extend \cite[Proposition 4.1]{CZ14} to sequences of
regular operators; see Theorem \ref{t-reg} and Corollary \ref{c-reg}.
\end{example}

\begin{example}[Subgradient projection]
\label{ex-subProj}%
\rm\ %
Let $f\colon \mathcal{H}\rightarrow \mathbb{R}$ be continuous and convex
with a nonempty sublevel set $S(f,0)$ and let $P_{f}:\mathcal{H}\rightarrow 
\mathcal{H}$ be a subgradient projection.

\begin{enumerate}
\item[(a)] If $f$ is Lipschitz continuous on bounded sets, then $P_{f}$ is
weakly regular \cite[Theorem 4.2.7]{Ceg12}. We recall that $f$ is Lipschitz
continuous on bounded sets if and only if $f$ maps bounded sets onto bounded
sets if and only if $\partial f$ is uniformly bounded on bounded sets \cite[%
Proposition 7.8]{BB96}. All three conditions hold true if $\mathcal{H}=%
\mathbb{R}^{n}$. If, in addition, $f$ is strongly convex, then $P_{f}$ is
boundedly regular. A detailed proof of this fact will be presented elsewhere.

\item[(b)] If $\mathcal{H}=\mathbb{R}^{n}$ then, by (a) and by the
equivalence of weak and strong convergence in a finite dimensional space, $%
P_{f}$ is boundedly regular. See also \cite[Lemma 24]{CZ13}.

\item[(c)] If $\mathcal{H}=\mathbb{R}^{n}$ and $f(z)<0$ for some $z\in 
\mathbb{R}^{n}$, then $P_{f}$ is boundedly linearly regular. Indeed, by (a)
and by Lemma \ref{l-Fuk}, for every compact $K\subseteq \mathbb{R}^{n}$,
there are $\delta ,\Delta >0$ such that $\Vert \partial f(x)\Vert \leq
\Delta $ and $\delta d(x,S(f,0))\leq f_{+}(x)$ for any $x\in K$. Thus, 
\begin{equation}
\Vert x-P_{f}x\Vert =\frac{f_{+}(x)}{\Vert g_{f}(x)\Vert }\geq \frac{\delta 
}{\Delta }d(x,S(f,0))\text{.}  \label{e-subProj}
\end{equation}%
Since $\limfunc{Fix}P_{f}=S(f,0)$, inequality (\ref{e-subProj}) proves the
bounded linear regularity of $P_{f}$.
\end{enumerate}
\end{example}

The operators presented in Examples \ref{ex-NE-BR} and \ref{ex-subProj}(b)
need not be boundedly regular if $\dim \mathcal{H}=\infty $ as the following
example shows.

\begin{example}[Subgradient projection which is not regular]
\rm\ %
Let $C_{1},C_{2}\subseteq \mathcal{H}$ be closed convex and $x^{0}\in 
\mathcal{H}$. Suppose that:

\begin{enumerate}
\item[(i)] $C:=C_{1}\cap C_{2}=\{0\}$,

\item[(ii)] $d(x^{0},C_{2})\leq d(x^{0},C_{1})$,

\item[(iii)] the sequence $\{x^{k}\}_{k=0}^{\infty }$ defined by the
recurrence $x^{k+1}=P_{C_{2}}P_{C_{1}}x^{k}$ converges weakly to $0$, but $%
\{x^{k}\}_{k=0}^{\infty }$ does not converge in norm.
\end{enumerate}

\noindent A construction of $C_{1},C_{2}$ and a point $x^{0}$, satisfying
(i)-(iii) is due to\ Hundal \cite{Hun04}; see also \cite{MR03}. Define a
function $f:\mathcal{H}\rightarrow 
\mathbb{R}
$ as follows: 
\begin{equation}
f(x)=\max \{d(x,C_{1}),d(x,C_{2})\}\text{.}
\end{equation}%
Clearly, $f$ is continuous and convex as the maximum of continuous and
convex functions. It is easy to check that for $x\in C_{1}\setminus C$ we
have $f(x)=d(x,C_{2})$ and $g_{f}(x)=\nabla f(x)=\frac{x-P_{C_{2}}x}{%
d(x,C_{2})}$, and that for $x\in C_{2}\setminus C$ we have $f(x)=d(x,C_{1})$
and $g_{f}(x)=\nabla f(x)=\frac{x-P_{C_{1}}x}{d(x,C_{1})}$. Let $u^{k}$ be
defined by the recurrence 
\begin{equation}
u^{k+1}=P_{f}u^{k}\text{,}
\end{equation}%
with $u^{0}=x^{0}$. Then we have 
\begin{equation}
u^{k+1}=\left\{ 
\begin{array}{ll}
P_{C_{1}}u^{k}\text{,} & \text{for }k=2n\text{,} \\ 
P_{C_{2}}u^{k}\text{,} & \text{for }k=2n+1\text{.}%
\end{array}%
\right.
\end{equation}%
By Hundal's construction, $u^{k}$ converges weakly to $0$ but does not
converge in norm, that is, $\limsup_{k}\Vert u^{k}\Vert >0$. Now it is
easily seen that $P_{f}$ is not boundedly regular. Indeed. $%
\{u^{k}\}_{k=0}^{\infty }$ is bounded as a weakly convergent sequence.
Moreover, $\lim_{k}\left\Vert u^{k}-P_{f}u^{k}\right\Vert =0$, because $%
P_{f} $ is SQNE (see Lemma \ref{f-6a}). But 
\begin{equation}
\limsup_{k}d(u^{k},C)=\limsup_{k}\Vert u^{k}\Vert >0\text{.}
\end{equation}%
Thus, $P_{f}$ is not boundedly regular.
\end{example}

\section{Regular sequences of operators}

\label{s-RS}

\begin{definition}
\label{d-RS}%
\rm\ %
Let $S\subseteq \mathcal{H}$ be nonempty, and $C\subseteq \mathcal{H}$ be
nonempty, closed and convex. We say that the sequence $\{U_{k}\}_{k=0}^{%
\infty }$ of quasi-nonexpansive operators $U_{k}\colon \mathcal{H}%
\rightarrow \mathcal{H}$ is:

\begin{enumerate}
\item[(a)] \textit{weakly }$C$-\textit{regular} over $S$ if for each $%
\{x^k\}_{k=0}^\infty\subseteq S$ and $\{n_k\}_{k=0}^\infty\subseteq
\{k\}_{k=0}^\infty$, we have 
\begin{equation}
\left . 
\begin{array}{l}
x^{n_k}\rightharpoonup y \\ 
\|U_k x^k-x^k\|\rightarrow 0%
\end{array}
\right\}\quad\Longrightarrow\quad y\in C;  \label{e-Wk}
\end{equation}

\item[(b)] $C$-\textit{regular} over $S$ if for any sequence $%
\{x^{k}\}_{k=0}^{\infty }\subseteq S$, we have 
\begin{equation}
\lim_{k\rightarrow \infty }\Vert U_{k}x^{k}-x^{k}\Vert =0\quad
\Longrightarrow \quad \lim_{k\rightarrow \infty }d(x^{k},C)=0\text{;}
\label{e-Rk}
\end{equation}

\item[(c)] \textit{linearly }$C$-\textit{regular} over $S$ if there is $%
\delta >0$ such that for all $x\in S$ and $k=0,1,2,\ldots$, we have 
\begin{equation}
d(x,C)\leq \delta \Vert U_{k}x-x\Vert \text{.}  \label{e-LRk}
\end{equation}%
The constant $\delta $ is called a \textit{modulus} of the linear $C$%
-regularity\textit{\ of }$\{U_{k}\}_{k=0}^{\infty }$ over $S$.
\end{enumerate}

\noindent If any of the above regularity conditions holds for $S=\mathcal{H}$%
, then we omit the phrase ``over $S$''. If any of the above regularity
conditions holds for every bounded subset $S\subseteq \mathcal{H}$, then we
precede the corresponding term with the adverb ``boundedly'' while omitting
the phrase ``over $S$''\ (we allow $\delta $ to depend on $S$ in (c)). We
say that $\{U_{k}\}_{k=0}^{\infty }$ is \textit{(boundedly) weakly regular,
regular or linearly regular} (over $S$), if 
\begin{equation}
C=\bigcap_{k=0}^{\infty }\limfunc{Fix}U_{k}\neq\emptyset
\end{equation}
in (a), (b) or (c), respectively.
\end{definition}

Setting $U_{k}=U$ for all $k\geq 0$ in Definition \ref{d-RS}, we arrive at
Definition \ref{d-R} of a (weakly, linearly) $C$-regular operator. Although
all the three sets $C$, $F:=\bigcap_{k=0}^\infty \limfunc{Fix} U_k$ and $S$
are not formally related in the above definition, similarly to the case of a
single operator, the most common setting that we are interested in is where 
\begin{equation}
C=F, \quad S\cap F \neq \emptyset \quad \text{and} \quad S \text{ is bounded}%
.
\end{equation}
We now adjust Remark \ref{r-WR} to the case of a sequence of operators.

\begin{remark}[Weak regularity]
\label{r-WR2}%
\rm\ %
Observe that:

\begin{enumerate}
\item[(i)] If $\{U_k\}_{k=0}^\infty$ is weakly $C$-regular over $S$ then,
obviously, for any sequence $\{x^{k}\}_{k=0}^{\infty }\subseteq S$, we have 
\begin{equation}  \label{e-WR2}
\left . 
\begin{array}{l}
x^k\rightharpoonup y \\ 
\|U_k x^k-x^k\|\rightarrow 0%
\end{array}
\right\}\quad\Longrightarrow\quad y\in C.
\end{equation}
The above condition \eqref{e-WR2} is no longer equivalent to \eqref{e-Wk},
as it was in the case of a constant sequence of operators. To see this,
following \cite[Sec. 4]{Ceg15}, we consider $U_{2k}:=T$ and $U_{2k+1}:=V$, $%
k=0,1,2,\ldots$, where $T,V\colon\mathcal{H }\rightarrow\mathcal{H}$ have a
nonempty common fixed point set $C=\limfunc{Fix}T\cap\limfunc{Fix}V$. Assume
that $V$ and $T$ are weakly regular. Then, clearly, $\{U_k\}_{k=0}^\infty$
satisfies \eqref{e-WR2}. Assume now that there is $y\in \limfunc{Fix}
T\setminus \limfunc{Fix} V$. Then, by taking $z\in \limfunc{Fix} V$ and
setting $x^{2k}=y$, $x^{2k+1}=z$, we see that $y$ is a weak cluster point of 
$\{x^k\}_{k=0}^\infty$ and $\|U_kx^k-x^k\|= 0$, but $y\notin \limfunc{Fix}
T\cap \limfunc{Fix} V$. Consequently, $\{U_k\}_{k=0}^\infty$ is not weakly
regular.

\item[(ii)] Assume that $F\neq \emptyset $. Then $\{U_{k}\}_{k=0}^{\infty }$
is boundedly weakly $C$-regular if and only if $\{U_{k}\}_{k=0}^{\infty }$
is weakly $C$-regular. Indeed, assume that $\{U_{k}\}_{k=0}^{\infty }$ is
boundedly weakly regular and let $\{x^{k}\}_{k=0}^{\infty }$ be such that $%
\Vert U_{k}x^{k}-x^{k}\Vert \rightarrow 0$ and $x^{n_{k}}\rightharpoonup y$.
Then, for any $z\in F$, the sequence 
\begin{equation}
y^{n}:=\left\{ 
\begin{array}{ll}
x^{k}\text{,} & \text{if }n=n_{k}\text{ for some }k\text{,} \\ 
z\text{,} & \text{otherwise}%
\end{array}%
\right.
\end{equation}%
is bounded, $y^{n_{k}}\rightharpoonup y$ and $\Vert U_{k}y^{k}-y^{k}\Vert
\rightarrow 0$. Consequently, by the bounded weak $C$-regularity of $%
\{U_{k}\}_{k=0}^{\infty }$, we have $y\in C$. This shows that $%
\{U_{k}\}_{k=0}^{\infty }$ is weakly regular. Therefore again, as it was for
the case of a single operator, there is no need to distinguish between
boundedly weakly ($C$-)regular and weakly ($C$-)regular sequences of
operators whenever $F\neq \emptyset $.
\end{enumerate}
\end{remark}

\begin{theorem}
\label{t-reg} Let $U_{k}\colon \mathcal{H}\rightarrow \mathcal{H}$ be
quasi-nonexpansive, $k=0,1,2,\ldots$, let $S\subseteq \mathcal{H}$ be
nonempty and let $C\subseteq \mathcal{H}$ be nonempty, closed and convex.
Then the following statements hold true:

\begin{enumerate}
\item[$\mathrm{(i)}$] If $\{U_{k}\}_{k=0}^{\infty }$ is linearly $C$-\textit{%
regular} over $S$, then $\{U_{k}\}_{k=0}^{\infty }$ is $C$-\textit{regular}
over $S$.

\item[$\mathrm{(ii)}$] If $\{U_{k}\}_{k=0}^{\infty }$ is $C$-regular over $S$%
, then $\{U_{k}\}_{k=0}^{\infty }$ is weakly $C$-regular over $S$.

\item[$\mathrm{(iii)}$] If $\{U_{k}\}_{k=0}^{\infty }$ is weakly $C$-regular
over $S$, $\mathcal{H}=\mathbb{R}^n$ and $S$ is bounded, then $%
\{U_{k}\}_{k=0}^{\infty }$ is $C$-regular over $S$.
\end{enumerate}
\end{theorem}

\begin{proof}
Part (i) follows directly from Definition \ref{d-RS}.

(ii) Suppose that $\{U_{k}\}_{k=0}^{\infty }$ is $C$-regular over $S$. Let $%
\{x^{k}\}_{k=0}^{\infty }\subseteq S$, $y$ be a weak cluster point of $%
\{x^{k}\}_{k=0}^{\infty }$ and $\Vert U_{k}x^{k}-x^{k}\Vert \rightarrow 0$.
Then $\lim_{k}d(x^{k},C)=0$. Let $\{x^{n_{k}}\}_{k=0}^{\infty }\subseteq
\{x^{k}\}_{k=0}^{\infty }$ be a subsequence converging weakly to $y$. By the
weak lower semicontinuity of $d(\cdot ,C)$, we have%
\begin{equation}
0=\lim_{k}d(x^{k},C)=\lim_{k}d(x^{n_{k}},C)\geq d(y,C)\geq 0\text{.}
\end{equation}%
Now the closedness of $C$ yields $y\in C$, which proves the weak $C$%
-regularity of $\{U_{k}\}_{k=0}^{\infty }$ over $S$.

(iii) Suppose that $\{U_{k}\}_{k=0}^{\infty }$ is weakly $C$-regular over $S$
and $\mathcal{H}=\mathbb{R}^n$. Let $\{x^{k}\}_{k=0}^{\infty }\subseteq S$
and $\lim_{k}\Vert U_{k}x^{k}-x^{k}\Vert =0$. We prove that $%
\lim_{k}d(x^{k},C)=0$. By the boundedness of $S$, there is a subsequence $%
\{x^{n_{k}}\}_{k=0}^{\infty }\subseteq\{x^{k}\}_{k=0}^{\infty }$ which
converges to $y\in \mathcal{H}$ and such that $\limsup_{k}d(x^{k},C)=%
\lim_{k}d(x^{n_{k}},C)$. The weak $C$-regularity of $\{U_{k}\}_{k=0}^{\infty
}$ over $S$ yields that $y\in C$. The continuity of $d(\cdot ,C)$ implies
that 
\begin{equation}
0\leq \limsup_{k}d(x^{k},C)=\lim_{k}d(x^{n_{k}},C)=d(y,C)=0\text{.}
\end{equation}%
Thus $\lim_{k}d(x^{k},C)=0$, that is, $\{U_{k}\}_{k=0}^{\infty }$ is $C$%
-regular over $S$.
\end{proof}

\bigskip

Parts (ii) and (iii) of Theorem \ref{t-reg} in the case $U_{k}=U$, $%
k=0,1,2\ldots$, were proved in \cite[Theorem 4.1]{CZ14}.

\begin{corollary}
\label{c-reg} Let $U_{k}\colon \mathcal{H}\rightarrow \mathcal{H}$ be
quasi-nonexpansive, $k=0,1,2,\ldots$, and assume that $\bigcap_{k=0}^\infty 
\limfunc{Fix}U_k\neq\emptyset$. Then the following statements hold true:

\begin{enumerate}
\item[$\mathrm{(i)}$] If $\{U_{k}\}_{k=0}^{\infty }$ is boundedly linearly
regular, then $\{U_{k}\}_{k=0}^{\infty }$ is boundedly regular.

\item[$\mathrm{(ii)}$] If $\{U_{k}\}_{k=0}^{\infty }$ is boundedly regular,
then $\{U_{k}\}_{k=0}^{\infty }$ is weakly regular.

\item[$\mathrm{(iii)}$] If $\{U_{k}\}_{k=0}^{\infty }$ is weakly regular and 
$\mathcal{H}=\mathbb{R}^{n}$, then $\{U_{k}\}_{k=0}^{\infty }$ is boundedly
regular.
\end{enumerate}
\end{corollary}

\begin{remark}
\label{r-RSa}%
\rm\ %
Let $\{U_{k}\}_{k=0}^{\infty }$ be a sequence of quasi-nonexpansive
operators and let $C\subseteq \mathcal{H}$ be nonempty, closed and convex.
Clearly, the sequence $\{U_{k}\}_{k=0}^{\infty }$ is (weakly, linearly) $C$%
-regular over any nonempty bounded subset $S\subseteq C$ because if $x\in C $%
, then $d(x,C)=0$. Let $S_{i}\subseteq \mathcal{H}$, $i=1,2$, be nonempty.
If $\{U_{k}\}_{k=0}^{\infty }$ is (weakly, linearly) $C$-regular over $S_{i}$%
, $i=1,2$, then $\{U_{k}\}_{k=0}^{\infty }$ is (weakly, linearly) $C$%
-regular over $S:=S_{1}\cup S_{2}$. Thus, without loss of generality, we can
add to $S$ an arbitrary bounded subset of $C$. Moreover, if $%
\{U_{k}\}_{k=0}^{\infty }$ is (weakly, linearly) $C$-regular over $S$, then $%
\{U_{k}\}_{k=0}^{\infty }$ is (weakly, linearly) $C$-regular over an
arbitrary nonempty subset of $S$. Thus in the definition of boundedly
(weakly, linearly) $C$-regular sequences of operators we can restrict the
bounded subsets $S$ to balls $B(z,R)$, where $z\in C$ is fixed and $R>0$.
\end{remark}

\begin{remark}
\label{r-RS}%
\rm\ %
Let $C_{1},C_{2}$ $\subseteq \mathcal{H}$ be nonempty, closed and convex, $%
C_{1}\subseteq C_{2}$, and $S\subseteq \mathcal{H}$ be nonempty. Clearly, $%
\{U_{k}\}_{k=0}^{\infty }$ is (weakly, linearly) $C_{2}$-regular over $S$ if 
$\{U_{k}\}_{k=0}^{\infty }$ is (weakly, linearly) $C_{1}$-regular over $S$.
\end{remark}

We finish this section with two natural properties of (weakly, linearly) $C$%
-regular sequences of operators.

\begin{proposition}[Relaxation]
\label{l-rel-BR} Let $T_{k}\colon \mathcal{H}\rightarrow \mathcal{H}$ be
quasi-nonexpansive, $k=0,1,2,\ldots $, let $S\subseteq \mathcal{H}$ be
nonempty and let $C\subseteq \mathcal{H}$ be nonempty, closed and convex.
Suppose that $\{T_{k}\}_{k=0}^{\infty }$ is weakly (boundedly, boundedly
linearly) $C$-regular over $S$ (with modulus $\delta $) and $U_{k}:=\limfunc{%
Id}+\lambda _{k}(T_{k}-\limfunc{Id})$, where $0<\lambda =\inf_{k}\lambda
_{k}\leq \lambda _{k}\leq 1$. Then the sequence $\{U_{k}\}_{k=0}^{\infty }$
is weakly, (boundedly, boundedly linearly) $C$-regular over $S$ (with
modulus $\delta /\lambda $).
\end{proposition}

\begin{proof}
The lemma follows directly from Definition \ref{d-RS}.
\end{proof}

\begin{proposition}[Subsequences of regular operators]
\label{l-sub-BR} Let $U_{k}\colon \mathcal{H}\rightarrow \mathcal{H}$ be
quasi-nonexpansive, $k=0,1,2,\ldots$, let $S\subseteq \mathcal{H}$ be
nonempty and let $C\subseteq \mathcal{H}$ be nonempty, closed and convex.
Moreover, let $F:=\bigcap_{k=0}^{\infty }\limfunc{Fix}U_{k}$. Then the
following statements hold true:

\begin{enumerate}
\item[$\mathrm{(i)}$] If $\{U_{k}\}_{k=0}^{\infty }$ is weakly $C$-regular
over $S$, then any of its subsequences $\{U_{n_{k}}\}_{k=0}^{\infty }$ is
weakly $C$-regular over $S$, whenever $S\cap F\neq\emptyset$.

\item[$\mathrm{(ii)}$] If $\{U_{k}\}_{k=0}^{\infty }$ is $C$-regular over $S$%
, then any of its subsequences $\{U_{n_{k}}\}_{k=0}^{\infty }$ is $C$%
-regular over $S$, whenever $S\cap F\neq \emptyset $.

\item[$\mathrm{(iii)}$] If $\{U_{k}\}_{k=0}^{\infty }$ is linearly $C$%
-regular over $S$ with modulus $\delta $, then any of its subsequences $%
\{U_{n_{k}}\}_{k=0}^{\infty }$ is linearly $C$-regular over $S$ with a
modulus $\delta $.
\end{enumerate}

Moreover, if $\{U_{k}\}_{k=0}^{\infty }$ is weakly, boundedly or boundedly
linearly regular, then $F=\bigcap_k \limfunc{Fix} U_{n_k}$ for every
subsequence $\{n_k\}_{k=0}^\infty\subseteq \{k\}_{k=0}^\infty$.
\end{proposition}

\begin{proof}
(i) Suppose that $\{U_{k}\}_{k=0}^{\infty }$ is weakly $C$-regular over $S$.
Let $\{x^{k}\}_{k=0}^{\infty }\subseteq S$, $\lim_{k}\Vert
U_{n_{k}}x^{k}-x^{k}\Vert =0$, $y$ be a weak cluster point of $%
\{x^{k}\}_{k=0}^{\infty }$ and $\{x^{m_{k}}\}_{k=0}^{\infty }\subseteq
\{x^{k}\}_{k=0}^{\infty }$ be a subsequence converging weakly to $y$. We
claim that $y\in C$. To show this, let $z\in S\cap F$ and define 
\begin{equation}
y^{n}:=\left\{ 
\begin{array}{ll}
x^{m_{k}}\text{,} & \text{if }n=n_{m_{k}}\text{ for some }k\text{,} \\ 
z\text{,} & \text{otherwise.}%
\end{array}%
\right.
\end{equation}%
Then $\{y^{n}\}_{n=0}^{\infty }\subseteq S$ and moreover, we have 
\begin{equation}
\Vert U_{n}y^{n}-y^{n}\Vert =\left\{ 
\begin{array}{ll}
\Vert U_{n_{m_{k}}}x^{m_{k}}-x^{m_{k}}\Vert \text{,} & \text{if }n=n_{m_{k}}%
\text{ for some }k\text{,} \\ 
0\text{,} & \text{otherwise.}%
\end{array}%
\right.
\end{equation}%
By assumption, $\lim_{k}\Vert U_{n_{m_{k}}}x^{m_{k}}-x^{m_{k}}\Vert =0$.
Consequently, $\lim_{n}\Vert U_{n}y^{n}-y^{n}\Vert =0$. Since $%
\{U_{n}\}_{n=0}^{\infty }$ is weakly $C$-regular over $S$ and $y$ is a weak
cluster point of $\{y^{n}\}_{n=0}^{\infty }$, we have $y\in C$.

(ii) Suppose that $\{U_{k}\}_{k=0}^{\infty }$ is $C$-regular over $S$. Let $%
\{x^{k}\}_{k=0}^{\infty }\subseteq S$, $\lim_{k}\Vert
U_{n_{k}}x^{k}-x^{k}\Vert =0$. We claim that $\lim_{k}d(x^{k},C)=0$. To show
this, let $z\in S\cap F\cap C$ and define 
\begin{equation}
y^{n}:=\left\{ 
\begin{array}{ll}
x^{k}\text{,} & \text{if }n=n_{k}\text{ for some }k\text{,} \\ 
z\text{,} & \text{otherwise.}%
\end{array}%
\right.
\end{equation}%
Then, as in (i), $\{y^n\}_{n=0}^\infty\subseteq S$ and we have 
\begin{equation}
\Vert U_{n}y^{n}-y^{n}\Vert =\left\{ 
\begin{array}{ll}
\Vert U_{n_{k}}x^{k}-x^{k}\Vert \text{,} & \text{if }n=n_{k}\text{,} \\ 
0\text{,} & \text{otherwise.}%
\end{array}%
\right.
\end{equation}%
By assumption, $\lim_{k}\Vert U_{n_{k}}x^{k}-x^{k}\Vert =0$. Consequently, $%
\lim_{n}\Vert U_{n}y^{n}-y^{n}\Vert =0$. Since $\{U_{n}\}_{n=0}^{\infty }$
is $C$-regular over $S$, we have $\lim_{n}d(y^{n},C)=0$, which yields $%
\lim_{k}d(x^{k},C)=0$.

The proof of part (iii) is straightforward.

Assume that $\{U_{k}\}_{k=0}^{\infty }$ is weakly regular, $%
\{n_{k}\}_{k=0}^{\infty }\subseteq \{k\}_{k=0}^{\infty }$ and let $z\in
\bigcap_{k}\limfunc{Fix}U_{n_{k}}$. We show that $z\in F$. Define $y^{k}=z$
for all $k=0,1,2,\ldots $. Then, by (i), we see that $z$ has to be in $F$.
Since bounded and bounded linear regularity imply weak regularity (Corollary %
\ref{c-reg}), the proof is complete.
\end{proof}

A variant of part (i), as well as the last statement from the above
proposition, were observed in \cite[Lemma 4.6, Remark 4.7]{Ceg15}.

\section{Convex combinations and products of regular sequences of operators 
\label{s-4}}

Theorems \ref{t-Rk1} and \ref{t-Rk2} below show that a family of (weakly,
linearly) regular sequences of operators having a common fixed point is
closed under convex combinations and compositions. We consider here $p$
sequences of operators $\{U_{j}^{k}\}_{k=0}^{\infty }$, $j=1,2,...,p$, and $%
m $ sets $C_{i}$, $i=1,2,...,m$.

\begin{theorem}
\label{t-Rk1} For each $k=0,1,2,\ldots ,$ let $U_{k}:=\sum_{j=1}^{p}\omega
_{j}^{k}U_{j}^{k}$, where $U_{j}^{k}\colon \mathcal{H}\rightarrow \mathcal{H}
$ is $\rho _{j}^{k}$-strongly quasi-nonexpansive, $\rho _{j}^{k}\geq 0$, $%
\omega _{j}^{k}\geq 0$, $j\in J:=\{1,\ldots ,p\}$, $\sum_{j\in J}\omega
_{j}^{k}=1$. Moreover, for each $i\in I:=\{1,\ldots ,m\}$, let $%
C_{i}\subseteq \mathcal{H}$ be closed and convex. Moreover, let $S\subseteq 
\mathcal{H}$ be bounded, $F_{0}:=\bigcap_{j\in J}\bigcap_{k\geq 0}\limfunc{%
Fix}U_{j}^{k}$, $C:=\bigcap_{i\in I}C_{i}$ and assume that $C\subseteq F_{0}$
is nonempty.

\begin{enumerate}
\item[$\mathrm{(i)}$] Suppose that for some $i\in I$, there is $%
\{j_{k}\}_{k=0}^{\infty }\subseteq J$ such that the sequence $%
\{U_{j_{k}}^{k}\}_{k=0}^{\infty }$ is weakly $C_{i}$-regular over $S$ and $%
\sigma_i:=\inf_k\omega_{j_k}^k\rho_{j_k}^k>0$. Then the sequence $%
\{U_{k}\}_{k=0}^{\infty }$ is weakly $C_{i}$-regular over $S$. If this
property holds for all $i\in I$, then $\{U_{k}\}_{k=0}^{\infty }$ is weakly $%
C$-regular over $S$;

\item[$\mathrm{(ii)}$] Suppose that for some $i\in I$, there is $%
\{j_{k}\}_{k=0}^{\infty }\subseteq J$ such that the sequence $%
\{U_{j_{k}}^{k}\}_{k=0}^{\infty }$ is $C_{i}$-regular over $S$ and $%
\sigma_i:=\inf_k\omega_{j_k}^k\rho_{j_k}^k>0$. Then the sequence $%
\{U_{k}\}_{k=0}^{\infty }$ is $C_{i}$-regular over $S$. If the property
holds for all $i\in I$ and $\{C_i\mid i\in I\}$ is regular over $S$, then $%
\{U_{k}\}_{k=0}^{\infty }$ is $C$-regular over $S$.

\item[$\mathrm{(iii)}$] Suppose that for any $i\in I$, there is $%
\{j_{k}\}_{k=0}^{\infty }\subseteq J$ such that the sequence $%
\{U_{j_{k}}^{k}\}_{k=0}^{\infty }$ is linearly $C_{i}$-regular over $S$ with
modulus $\delta _{i}$, $\sigma_i:=\inf_k\omega_{j_k}^k\rho_{j_k}^k>0$ and $%
\{C_i\mid i\in I\}$ is linearly regular over $S$ with modulus $\kappa >0$.
Then $\{U_{k}\}_{k=0}^{\infty }$ is linearly $C$-regular over $S$ with
modulus $2\kappa ^{2}\delta ^{2}/\sigma$, where $\sigma:=\min_i\sigma_i$ and 
$\delta :=\min_{i\in I}\delta _{i}$.
\end{enumerate}
\end{theorem}

\begin{proof}
Let $z\in C$ and $\{x^{k}\}_{k=0}^{\infty }\subseteq S$. By Fact \ref{f-5}%
(iii), for any $k=0,1,2,\ldots,$ and $j\in J$ we have 
\begin{equation}
\frac{\omega _{j}^{k}\rho _{j}^{k}}{2R}\Vert U_{j}^{k}x^{k}-x^{k}\Vert
^{2}\leq \frac{1}{2R}\sum_{i=1}^{p}\omega _{i}^{k}\rho _{i}^{k}\Vert
U_{i}^{k}x^{k}-x^{k}\Vert ^{2}\leq \Vert U_{k}x^{k}-x^{k}\Vert \text{,}
\label{e-1}
\end{equation}%
where $R>0$ is such that $S\subseteq B(z,R)$.

(i) Let $y$ be a weak cluster point of $\{x^{k}\}_{k=0}^{\infty }$, $i\in I$
and $\{j_{k}\}_{k=0}^{\infty }\subseteq J$ be such that the sequence $%
\{U_{j_{k}}^{k}\}_{k=0}^{\infty }$ is weakly $C_{i}$-regular over $S$.
Suppose that $\lim_{k}\Vert U_{k}x^{k}-x^{k}\Vert =0$. Inequalities (\ref%
{e-1}) with $j=j_{k}$, $k=0,1,2,\ldots,$ and the inequality $\sigma_i>0$
yield $\lim_{k}\Vert U_{j_{k}}^{k}x^{k}-x^{k}\Vert =0$. Thus $y\in C_{i}$,
that is, $\{U_{k}\}_{k=0}^{\infty }$ is weakly $C_{i}$-regular over $S$. If
this property holds for all $i\in I$, then $y\in C$, that is, $%
\{U_{k}\}_{k=0}^{\infty }$ is weakly $C$-regular over $S$.

(ii) Let $i\in I$ and $\{j_{k}\}_{k=0}^{\infty }\subseteq J$ be such that $%
\{U_{j_{k}}^{k}\}_{k=0}^{\infty }$ is $C_{i}$-regular over $S$. Suppose that 
$\lim_{k}\Vert U_{k}x^{k}-x^{k}\Vert =0$. By (\ref{e-1}) with $j=j_{k}$, $%
k=0,1,2,\ldots,$ and since $\sigma_i>0$, we have $\lim_{k}\Vert
U_{j_{k}}^{k}x^{k}-x^{k}\Vert =0$. Consequently, $\lim_{k}d(x^{k},C_{i})=0$,
that is, $\{U_{k}\}_{k=0}^{\infty }$ is $C_{i}$-regular over $S$. The proof
of the second part of (ii) follows now directly from the definition of a
regular family of sets.

(iii) Let $i\in I$ be arbitrary and $\{j_{k}\}_{k=0}^{\infty }\subseteq J$
be such that the sequence $\{U_{j_{k}}^{k}\}_{k=0}^{\infty }$ is linearly $%
C_{i}$-regular over $S$ with modulus $\delta _{i}$. By (\ref{e-1}) with $%
j=j_{k}$, $x^{k}=x$, $z=P_{C}x$ and $R=\Vert x-z\Vert =d(x,C)$, we get 
\begin{equation}
\Vert U_{j_{k}}^{k}x-x\Vert ^{2}\leq \sum_{j\in J}\frac{\omega _{j}^{k}\rho
_{j}^{k}}{\sigma}\Vert U_{j}^{k}x-x\Vert ^{2}\leq \frac{2d(x,C)}{\sigma}%
\Vert U_{k}x-x\Vert  \label{e-1a}
\end{equation}%
for all $x\in S$. Since the sequence $\{U_{j_{k}}^{k}\}_{k=0}^{\infty }$ is
linearly $C_{i}$-regular over $S$ with modulus $\delta _{i}$, we also have $%
d(x,C_{i})\leq \delta _{i}\Vert U_{j_{k}}^{k}x-x\Vert $, $x\in S$, $%
k=0,1,2,\ldots,$ and thus, by (\ref{e-1a}), we arrive at 
\begin{equation}
d^{2}(x,C_{i})\leq \frac{2\delta _{i}^{2}d(x,C)}{\sigma}\Vert U_{k}x-x\Vert
\label{e-1b}
\end{equation}%
for all $x\in S$ and $k=0,1,2,\ldots$. Since $\{C_i\mid i\in I\}$ is
linearly regular over $S$ with modulus $\kappa $, we get%
\begin{equation}
d^{2}(x,C)\leq \kappa ^{2}d^{2}(x,C_{i})\leq \frac{2\kappa ^{2}\delta
_{i}^{2}d(x,C)}{\sigma}\Vert U_{k}x-x\Vert
\end{equation}%
for all $i\in I$, $x\in S$ and $k=0,1,2,\ldots$. This yields 
\begin{equation}
d(x,C)\leq \frac{2\kappa ^{2}\delta ^{2}}{\sigma}\Vert U_{k}x-x\Vert
\end{equation}%
for all $x\in S$ and $k=0,1,2,\ldots,$ which means that $\{U_{k}\}_{k=0}^{%
\infty }$ is linearly $C$-regular over $S$ with modulus $2\kappa ^{2}\delta
^{2}/\sigma$, as asserted.
\end{proof}

\begin{corollary}
\label{c-Rk1a} For each $k=0,1,2,\ldots ,$ let $U_{k}:=\sum_{i=1}^{m}\omega
_{i}^{k}U_{i}^{k}$, where $U_{i}^{k}\colon \mathcal{H}\rightarrow \mathcal{H}
$ is $\rho _{i}^{k}$-strongly quasi-nonexpansive, $i\in I:=\{1,\ldots ,m\}$.
Assume that $\rho :=\min_{i\in I}\inf_{k}\rho _{i}^{k}>0$, $\omega
:=\min_{i}\inf_{k}\omega _{i}^{k}>0$, $\sum_{i\in I}\omega _{i}^{k}=1$ and $%
F_{0}:=\bigcap_{i\in I}F_{i}\neq \emptyset $, where $F_{i}:=\bigcap_{k\geq 0}%
\limfunc{Fix}U_{i}^{k}$. Moreover, let $S\subseteq \mathcal{H}$ be bounded.

\begin{enumerate}
\item[$\mathrm{(i)}$] Suppose that for any $i\in I$, $\{U_{i}^{k}\}_{k=0}^{%
\infty }$ is weakly regular over $S$. Then the sequence $\{U_{k}\}_{k=0}^{%
\infty }$ is also weakly regular over $S$.

\item[$\mathrm{(ii)}$] Suppose that for any $i\in I$, $\{U_{i}^{k}\}_{k=0}^{%
\infty }$ is regular over $S$ and the family $\{F_i \mid i\in I\}$ is
regular over $S$. Then the sequence $\{U_{k}\}_{k=0}^{\infty }$ is also
regular over $S$.

\item[$\mathrm{(iii)}$] Suppose that for any $i\in I$, $\{U_{i}^{k}%
\}_{k=0}^{\infty }$ is linearly regular over $S$ with modulus $\delta _{i}$, 
$\delta :=\min_{i\in I}\delta _{i}>0$, and the family $\{F_i \mid i\in I\}$
is linearly regular over $S$ with modulus $\kappa >0$. Then the sequence $%
\{U_{k}\}_{k=0}^{\infty }$ is regular over $S$ with modulus $2\kappa
^{2}\delta ^{2}/(\omega \rho )$.
\end{enumerate}
\end{corollary}

\begin{proof}
It suffices to substitute $J=I$, $C_{i}=F_{i}$ and $j_{k}=i$, $%
k=0,1,2,\ldots $ in Theorem \ref{t-Rk1}.
\end{proof}

\begin{corollary}
\label{c-Rk1b} Let $U:=\sum_{i=1}^{m}\omega _{i}U_{i}$, where $U_{i}\colon 
\mathcal{H}\rightarrow \mathcal{H}$ is $\rho _{i}$-strongly
quasi-nonexpansive, $i\in I:=\{1,\ldots ,m\}$. Assume that $\rho
:=\min_{i\in I}\rho _{i}>0$, $\omega :=\min_{i}\omega _{i}>0$, $\sum_{i\in
I}\omega _{i}=1$ and $F_{0}:=\bigcap_{i\in I}\limfunc{Fix}U_{i}\neq
\emptyset $. Moreover, let $S\subseteq \mathcal{H}$ be bounded.

\begin{enumerate}
\item[$\mathrm{(i)}$] Suppose that for any $i\in I$, $U_{i}$ is weakly
regular over $S$. Then $U$ is also weakly regular over $S$.

\item[$\mathrm{(ii)}$] Suppose that for any $i\in I$, $U_{i}$ is regular
over $S$ and the family $\{\limfunc{Fix}U_{i}\mid i\in I\}$ is regular over $%
S$. Then $U$ is also regular over $S$.

\item[$\mathrm{(iii)}$] Suppose that for any $i\in I$, $U_{i}$ is linearly
regular over $S$ with modulus $\delta _{i}$, $\delta :=\min_{i\in I}\delta
_{i}>0$, and the family $\{\limfunc{Fix}U_{i}\mid i\in I\}$ is linearly
regular over $S$ with modulus $\kappa >0$. Then $U$ is linearly regular over 
$S$ with modulus $2\kappa ^{2}\delta ^{2}/(\omega \rho )$.
\end{enumerate}
\end{corollary}

\begin{proof}
It suffices to substitute $U_{i}^{k}=U_{i}$ and $\omega _{i}^{k}=\omega _{i}$
for all $k=0,1,2,\ldots ,$ and $i\in I$ in Corollary \ref{c-Rk1a}.
\end{proof}

\bigskip

Since $S\subseteq \mathcal{H}$ is an arbitrary nonempty and bounded subset
in Theorem \ref{t-Rk1} and in Corollaries \ref{c-Rk1a} and \ref{c-Rk1b},
these three results are also true for boundedly (weakly, linearly) ($C_{i}$%
-)regular sequences of operators.

Note that if an operator (or sequence of operators) is boundedly linearly
regular with modulus $\delta $, then the same property holds with any
modulus $\gamma >\delta $. Therefore, without any loss of generality, we can
restrict the analysis to boundedly linearly regular operators (or sequence
of operators) with modulus $\delta \geq 1$.

\begin{theorem}
\label{t-Rk2} For each $k=0,1,2,\ldots ,$ let $U_{k}:=U_{p}^{k}U_{p-1}^{k}%
\ldots U_{1}^{k}$, where $U_{j}^{k}\colon \mathcal{H}\rightarrow \mathcal{H}$
is $\rho _{j}^{k}$-strongly quasi-nonexpansive, $j\in J:=\{1,\ldots ,p\}$
and $\rho :=\min_{j\in J}\inf_{k}\rho _{j}^{k}>0$. Moreover, for each $i\in
I:=\{1,\ldots ,m\}$, let $C_{i}\subseteq \mathcal{H}$ be closed and convex.
Let $F_{0}:=\bigcap_{j\in J}\bigcap_{k\geq 0}\limfunc{Fix}U_{j}^{k}$, $%
C:=\bigcap_{i\in I}C_{i}$ and assume that $C\subseteq F_{0}$ is nonempty.
Moreover, let $S:=B(z,R)$ for some $z\in C$ and $R>0$.

\begin{enumerate}
\item[$\mathrm{(i)}$] Suppose that for some $i\in I$, there is $%
\{j_{k}\}_{k=0}^{\infty }\subseteq J$ such that the sequence $%
\{U_{j_{k}}^{k}\}_{k=0}^{\infty }$ is weakly $C_{i}$-regular over $S$. Then
the sequence $\{U_{k}\}_{k=0}^{\infty }$ is weakly $C_{i}$-regular over $S$.
If this property holds for all $i\in I$, then $\{U_{k}\}_{k=0}^{\infty }$ is
weakly $C$-regular over $S$.

\item[$\mathrm{(ii)}$] Suppose that for some $i\in I$, there is $%
\{j_{k}\}_{k=0}^{\infty }\subseteq J$ such that the sequence $%
\{U_{j_{k}}^{k}\}_{k=0}^{\infty }$ is $C_{i}$-regular over $S$. Then the
sequence $\{U_{k}\}_{k=0}^{\infty }$ is $C_{i}$-regular over $S$. If this
property holds for all $i\in I$ and $\{C_i\mid i\in I\}$ is regular over $S$%
, then $\{U_{k}\}_{k=0}^{\infty }$ is $C$-regular over $S$.

\item[$\mathrm{(iii)}$] Suppose that for any $i\in I$, there is $%
\{j_{k}\}_{k=0}^{\infty }\subseteq J$ such that the sequence $%
\{U_{j_{k}}^{k}\}_{k=0}^{\infty }$ is linearly $C_{i}$-regular over $S$ with
modulus $\delta _{i}\geq 1$, $\delta :=\min_{i\in I}\delta _{i}$, and $%
\{C_{i}\mid i\in I\}$ is linearly regular over $S$ with modulus $\kappa >0$.
Then $\{U_{k}\}_{k=0}^{\infty }$ is linearly $C$-regular over $S$ with
modulus $2p\kappa ^{2}\delta ^{2}/\rho $.
\end{enumerate}
\end{theorem}

\begin{proof}
Let $z\in C$ and $\{x^{k}\}_{k=0}^{\infty }\subseteq S$. Denote $%
Q_{j}^{k}:=U_{j}^{k}U_{j-1}^{k}...U_{1}^{k}$, $Q_{0}^{k}:=\limfunc{Id}$ and $%
x_{j}^{k}:=Q_{j}^{k}x^{k}$, $j\in J$, $k=0,1,2,\ldots$. Clearly, $%
Q_{j}^{k}=U_{j}^{k}Q_{j-1}^{k}$, $x_{0}^{k}=x^{k}$ and $%
x_{j}^{k}=U_{j}^{k}x_{j-1}^{k}$, $j\in J$. By Fact \ref{f-6}(iii), for any $%
j\in J$, we have 
\begin{equation}
0\leq \frac{\rho _{j}^{k}}{2R}\Vert U_{j}^{k}x_{j-1}^{k}-x_{j-1}^{k}\Vert
^{2}\leq \frac{1}{2R}\sum_{l=1}^{p}\rho _{l}^{k}\Vert
U_{l}^{k}x_{l-1}^{k}-x_{l-1}^{k}\Vert ^{2}\leq \Vert U_{k}x^{k}-x^{k}\Vert 
\text{,}  \label{e-2}
\end{equation}%
$j\in J$. Suppose that $\lim_{k}\Vert U_{k}x^{k}-x^{k}\Vert =0$. Note that
the assumption $C\neq \emptyset $ and the quasi-nonexpansivity of $U_{l}^{k}$%
, $l\in J$, imply $\{x_{l}^{k}\}_{k=0}^{\infty }\subseteq S$, $l\in J$.
Inequalities (\ref{e-2}) and the inequality $\rho >0$ yield 
\begin{equation}
\lim_{k}\Vert U_{l}x_{l-1}^{k}-x_{l-1}^{k}\Vert =0
\end{equation}%
for all $l\in J$. For a sequence $\{j_{k}\}_{k=0}^{\infty }$, denote $%
y^{k}=x_{j_{k}-1}^{k}$. Clearly, $\{y^{k}\}_{k=0}^{\infty }\subseteq S$.

(i) Let $y$ be a weak cluster point of $\{x^{k}\}_{k=0}^{\infty }$, $i\in I$
and $\{j_{k}\}_{k=0}^{\infty }\subseteq J$ be such that the sequence $%
\{U_{j_{k}}^{k}\}_{k=0}^{\infty }$ is weakly $C_{i}$-regular over $S$.
Suppose that $\lim_{k}\Vert U_{k}x^{k}-x^{k}\Vert =0$. Inequalities (\ref%
{e-2}) with $j=j_{k}$, $k=0,1,2,\ldots,$ and the inequality $\rho >0$ yield $%
\lim_{k}\Vert U_{j_{k}}^{k}y^{k}-y^{k}\Vert =0$. Thus $y\in C_{i}$, that is, 
$\{U_{k}\}_{k=0}^{\infty }$ is weakly $C_{i}$-regular over $S$. If this
property holds for all $i\in I$, then $y\in C$, that is, $%
\{U_{k}\}_{k=0}^{\infty }$ is weakly $C$-regular over $S$.

(ii) Let $i\in I$ and $\{j_{k}\}_{k=0}^{\infty }\subseteq J$ be such that
the sequence $\{U_{j_{k}}^{k}\}_{k=0}^{\infty }$ is weakly $C_{i}$-regular
over $S$. Since $\rho >0$, inequalities (\ref{e-2}) yield 
\begin{equation}
\lim_{k}\Vert U_{j_{k}}^{k}y^{k}-y^{k}\Vert =0\text{.}  \label{e-Uik1}
\end{equation}%
By the $C_{i}$-regularity of $\{U_{j_{k}}^{k}\}_{k=0}^{\infty }$ over $S$
and by (\ref{e-Uik1}), we have 
\begin{equation}
\lim_{k}d(y^{k},C_{i})=0\text{.}  \label{e-dyjk1}
\end{equation}%
The definition of the metric projection and the triangle inequality yield%
\begin{eqnarray}
d(x^{k},C_{i}) &=&\Vert x^{k}-P_{C_{i}}x^{k}\Vert \leq \Vert
x^{k}-P_{C_{i}}y^{k}\Vert =\Vert
\sum_{j=0}^{j_{k}-2}(x_{j}^{k}-x_{j+1}^{k})+(y^{k}-P_{C_{i}}y^{k})\Vert 
\notag \\
&\leq &\sum_{j=0}^{j_{k}-2}\left\Vert x_{j}^{k}-x_{j+1}^{k}\right\Vert
+\left\Vert y^{k}-P_{C_{i}}y^{k}\right\Vert \leq \sum_{j=0}^{p-1}\Vert
U_{j+1}^{k}x_{j}^{k}-x_{j}^{k}\Vert +d(y^{k},C_{i})\text{.}
\end{eqnarray}%
By (\ref{e-2}), the inequality $\rho >0$ and the assumption that $%
\lim_{k}\Vert U_{k}x^{k}-x^{k}\Vert =0$, we have 
\begin{equation}
\lim_{k}\sum_{j=0}^{p-1}\Vert U_{j+1}^{k}x_{j}^{k}-x_{j}^{k}\Vert =0.
\end{equation}
This together with (\ref{e-dyjk1}) leads to $\lim_{k}d(x^{k},C_{i})=0$, that
is, $\{U_{k}\}_{k=0}^{\infty }$ is boundedly $C_{i}$-regular. The proof of
the second part of (ii) follows directly from the definition of a regular
family of sets.

(iii) Let $i\in I$ be arbitrary and $\{j_{k}\}_{k=0}^{\infty }\subseteq J$
be such that the sequence $\{U_{j_{k}}^{k}\}_{k=0}^{\infty }$ is linearly $%
C_{i}$-regular over $S$. Let $x\in S$. By (\ref{e-2}) with $x^{k}=x$, $%
z=P_{C}x$ and $R=\Vert x-z\Vert =d(x,C)$, we get 
\begin{equation}
\sum_{j=1}^{p}\Vert x_{j}^{k}-x_{j-1}^{k}\Vert ^{2}\leq \frac{2d(x,C)}{\rho }%
\Vert U_{k}x-x\Vert \text{.}  \label{e-2a}
\end{equation}%
By the linear $C_{i}$-regularity of $\{U_{j_{k}}^{k}\}_{k=0}^{\infty }$ over 
$S$ with modulus $\delta _{i}$, 
\begin{equation}
d(y^{k},C_{i})\leq \delta _{i}\Vert U_{j_{k}}^{k}y^{k}-y^{k}\Vert
\label{e-2b}
\end{equation}%
for all $k=0,1,2,\ldots $. By the definition of the metric projection, the
triangle inequality, inequality (\ref{e-2b}) and the assumption that $\delta
_{i}\geq 1$, we have%
\begin{eqnarray}
d^{2}(x,C_{i}) &\leq &\Vert x-P_{C_{i}}y^{k}\Vert ^{2}\leq \left(
\sum_{l=1}^{j_{k}-1}\Vert x_{l}^{k}-x_{l-1}^{k}\Vert +\Vert
y^{k}-P_{C_{i}}y^{k}\Vert \right) ^{2}  \notag \\
&=&\left( \sum_{l=1}^{j_{k}-1}\Vert x_{l}^{k}-x_{l-1}^{k}\Vert
+d(y^{k},C_{i})\right) ^{2}\leq \left( \sum_{l=1}^{j_{k}-1}\Vert
x_{l}^{k}-x_{l-1}^{k}\Vert +\delta _{i}\Vert U_{j_{k}}^{k}y^{k}-y^{k}\Vert
\right) ^{2}  \notag \\
&\leq &\delta _{i}^{2}\left( \sum_{l=1}^{j_{k}-1}\Vert
x_{l}^{k}-x_{l-1}^{k}\Vert +\Vert U_{j_{k}}^{k}y^{k}-y^{k}\Vert \right)
^{2}=\delta _{i}^{2}\left( \sum_{l=1}^{j_{k}}\Vert
x_{l}^{k}-x_{l-1}^{k}\Vert \right) ^{2}  \notag \\
&\leq &\delta _{i}^{2}\left( \sum_{l=1}^{p}\Vert x_{l}^{k}-x_{l-1}^{k}\Vert
\right) ^{2}\text{,}
\end{eqnarray}%
$x\in S$, $k=0,1,2,\ldots $. The above inequalities and the Cauchy-Schwarz
inequality $\langle e,a\rangle ^{2}\leq p\Vert a\Vert ^{2}$ with $%
e=(1,1,...,1)\in 
\mathbb{R}
^{p}$ and $a=(a_{1},a_{2},...a_{p})\in 
\mathbb{R}
^{p}$, where $a_{j}:=\Vert x_{j}^{k}-x_{j-1}^{k}\Vert $, $j\in J$, yield 
\begin{equation}
d^{2}(x,C_{i})\leq p\delta _{i}^{2}\sum_{l=1}^{p}\Vert
x_{l}^{k}-x_{l-1}^{k}\Vert ^{2}\text{,}  \label{e-2c}
\end{equation}%
$x\in S$, $k=0,1,2,\ldots $. Now the linear bounded regularity of $%
\{C_{i}\mid \in I\}$ with modulus $\kappa $, (\ref{e-2a}) and (\ref{e-2c})
imply that 
\begin{equation}
d^{2}(x,C)\leq \kappa ^{2}d^{2}(x,C_{i})\leq \frac{2p\delta _{i}^{2}\kappa
^{2}}{\rho }d(x,C)\Vert U_{k}x-x\Vert \text{,}
\end{equation}%
$x\in S$, $k=0,1,2,\ldots $ for all $i\in I$. This gives 
\begin{equation}
d(x,C)\leq \frac{2p\delta ^{2}\kappa ^{2}}{\rho }\Vert U_{k}x-x\Vert \text{,}
\end{equation}%
$x\in S$, $k=0,1,2,\ldots ,$ that is, $\{U_{k}\}_{k=0}^{\infty }$ is
linearly $C$-regular over $S$ with modulus $2p\kappa ^{2}\delta ^{2}/\rho $.
\end{proof}

\begin{corollary}
\label{c-Rk2a} For each $k=0,1,2,\ldots ,$ let $U_{k}:=U_{m}^{k}U_{m-1}^{k}%
\ldots U_{1}^{k}$, where $U_{i}^{k}\colon \mathcal{H}\rightarrow \mathcal{H}$
is $\rho _{i}^{k}$-strongly quasi-nonexpansive, $i\in I:=\{1,\ldots ,m\}$.
Assume that $\rho :=\min_{i\in I}\inf_{k}\rho _{i}^{k}>0$ and $%
F_{0}:=\bigcap_{i\in I}F_{i}\neq \emptyset $, where $F_{i}:=\bigcap_{k\geq 0}%
\limfunc{Fix}U_{i}^{k}$. Moreover, let $S:=B(z,R)$ for some $z\in F_{0}$ and 
$R>0$.

\begin{enumerate}
\item[$\mathrm{(i)}$] Suppose that for any $i\in I$, $\{U_{i}^{k}\}_{k=0}^{%
\infty }$ is weakly regular over $S$. Then the sequence $\{U_{k}\}_{k=0}^{%
\infty }$ is also weakly regular over $S$.

\item[$\mathrm{(ii)}$] Suppose that for any $i\in I$, $\{U_{i}^{k}\}_{k=0}^{%
\infty }$ is regular over $S$ and the family $\{F_{i}\mid i\in I\}$ is
regular over $S$. Then the sequence $\{U_{k}\}_{k=0}^{\infty }$ is also
regular over $S$.

\item[$\mathrm{(iii)}$] Suppose that for any $i\in I$, $\{U_{i}^{k}%
\}_{k=0}^{\infty }$ is linearly regular over $S$ with modulus $\delta
_{i}\geq 1$, $\delta :=\min_{i\in I}\delta _{i}>0$, and the family $%
\{F_{i}\mid i\in I\}$ is linearly regular over $S$ with modulus $\kappa >0$.
Then the sequence $\{U_{k}\}_{k=0}^{\infty }$ is regular over $S$ with
modulus $2m\kappa ^{2}\delta ^{2}/\rho $.
\end{enumerate}
\end{corollary}

\begin{proof}
It suffices to substitute $J=I$, $C_{i}=F_{i}$ and $j_{k}=i$, $%
k=0,1,2,\ldots $ in Theorem \ref{t-Rk2}.
\end{proof}

\begin{corollary}
\label{c-Rk2b} Let $U:=U_{m}U_{m-1}\ldots U_{1}$, where $U_{i}\colon 
\mathcal{H}\rightarrow \mathcal{H}$ is $\rho _{i}$-strongly
quasi-nonexpansive, $i\in I:=\{1,\ldots ,m\}$. Assume that $\rho
:=\min_{i\in I}\rho _{i}>0$ and $F_{0}:=\bigcap_{i\in I}\limfunc{Fix}%
U_{i}\neq \emptyset $. Moreover, let $S:=B(z,R)$ for some $z\in F_{0}$ and $%
R>0$.

\begin{enumerate}
\item[$\mathrm{(i)}$] Suppose that for any $i\in I$, $U_{i}$ is weakly
regular over $S$. Then $U$ is also weakly regular over $S$.

\item[$\mathrm{(ii)}$] Suppose that for any $i\in I$, $U_{i}$ is regular
over $S$ and the family $\{\limfunc{Fix}U_{i}\mid i\in I\}$ is regular over $%
S$. Then $U$ is also regular over $S$.

\item[$\mathrm{(iii)}$] Suppose that for any $i\in I$, $U_{i}$ is linearly
regular over $S$ with modulus $\delta _{i}\geq 1$, $\delta :=\min_{i\in
I}\delta _{i}>0$, and the family $\{\limfunc{Fix}U_{i}\mid i\in I\}$ is
linearly regular over $S$ with modulus $\kappa >0$. Then $U$ is linearly
regular over $S$ with modulus $2m\kappa ^{2}\delta ^{2}/\rho $.
\end{enumerate}
\end{corollary}

\begin{proof}
It suffices to substitute $U_{i}^{k}=U_{i}$ for all $k=0,1,2,\ldots $ and $%
i\in I$ in Corollary \ref{c-Rk2a}.
\end{proof}

\begin{example}
\label{ex-BRk}%
\rm\ %
Let $S_{i}:\mathcal{H}\rightarrow \mathcal{H}$ be QNE, $C_{i}:=\limfunc{Fix}%
S_{i}$, $i\in I:=\{1,2,...,m\}$ and $C:=\bigcap_{i\in I}C_{i}\neq \emptyset $%
. Set $J:=\{1,2,...,p\}$ and $S_{i,\lambda }:=\limfunc{Id}+\lambda (S_i-%
\limfunc{Id})$, where $\lambda \in \lbrack 0,1]$.

\begin{enumerate}
\item[(a)] (\textit{Block iterative sequence}) Let $J_{k}\subseteq J$ be an
ordered subset, $k=0,1,2,\ldots $. Let 
\begin{equation}
T_{j}^{k}:=\limfunc{Id}+\lambda _{j}^{k}(\sum_{i\in I_{j}^{k}}\omega
_{ij}^{k}S_{i}-\limfunc{Id})=\sum_{i\in I_{j}^{k}}\omega
_{ij}^{k}S_{i,\lambda _{j}^{k}}\text{,}  \label{e-Tjk1}
\end{equation}%
where $I_{j}^{k}\subseteq I$, $\omega _{ij}^{k}\geq \delta >0$ for all $i\in
I_{j}^{k}$, $\sum_{i\in I_{j}^{k}}\omega _{ij}^{k}=1$, $j\in J_{k}$, $%
k=0,1,2,\ldots $, and 
\begin{equation}
T_{k}:=\prod_{j\in J_{k}}T_{j}^{k}\text{.}  \label{e-Tk1}
\end{equation}%
The block iterative methods for solving the consistent convex feasibility
problem \cite{Ceg12} can be represented in the form $x^{k+1}=T_{k}x^{k}$
with a sequence of operators $T_{k}$ given by (\ref{e-Tk1}), where $%
T_{j}^{k} $ are defined by (\ref{e-Tjk1}). Suppose that $\underline{\lambda }%
:=\inf_{k\geq 0}\min_{j\in J_{k}}\lambda _{j}^{k}>0$ and $\bar{\lambda}%
:=\sup_{k\geq 0}\max_{j\in J_{k}}\lambda _{j}^{k}<1$. Put $\rho
_{j}^{k}:=(1-\lambda _{j}^{k})/\lambda _{j}^{k}$, $j\in J_{k}$, $%
k=0,1,2,\ldots $, and $\rho :=\inf_{k\geq 0}\min_{j\in J_{k}}\rho _{j}^{k}$.
Then $S_{i,\lambda _{j}^{k}}$ is $\rho _{j}^{k}$-SQNE, $i\in I_{j}^{k}$ (see
Corollary \ref{c-1}(ii)). Clearly, $\rho _{j}^{k}\geq \rho \geq (1-\bar{%
\lambda})/\bar{\lambda}>0$ for all $j\in J_{k}$ and $k=0,1,2,\ldots $.
Suppose that $I^{k}:=\bigcup_{j\in J_{k}}I_{j}^{k}=I$ for all $%
k=0,1,2,\ldots $. Let $i\in I$ be arbitrary but fixed, and $j_{k}\in J_{k}$
and $i_{k}\in I_{j_{k}}^{k}$ be such that $i_{i_{k},j_{k}}=i$. By Facts \ref%
{f-2}(iv) and \ref{f-5}(i), $T_{j}^{k}$ is $\rho _{j}^{k}$-SQNE, where $\rho
_{j}^{k}\geq \rho >0$ for all $j\in J_{k}$ and $k=0,1,2,\ldots $. Moreover, $%
F_{0}:=\bigcap_{k\geq 0}\bigcap_{j\in J_{k}}\limfunc{Fix}T_{j}^{k}=C$.
Suppose that each $S_{i}$, $i\in I$, is weakly (boundedly, boundedly
linearly) regular. Then, by Proposition \ref{l-rel-BR}, the sequence $%
\{S_{i,\lambda _{j_{k}}^{k}}\}_{k=0}^{\infty }$ is weakly (boundedly,
boundedly linearly) $C_{i}$-regular. Let us now separately consider the
above three different types of regularity.

\begin{enumerate}
\item[(i)] Suppose first that each $S_{i}$, $i\in I$, is weakly regular.
Then, by Theorem \ref{t-Rk1} (i), the sequence $\{T_{j_{k}}^{k}\}_{k=0}^{%
\infty }$ is weakly $C_{i}$-regular and, consequently, by Theorem \ref{t-Rk2}
(i), the sequence $\{T_{k}\}_{k=0}^{\infty }$ is weakly $C_{i}$-regular.
Moreover, since $i\in I$ is arbitrary, the sequence $\{T_{k}\}_{k=0}^{\infty
}$ is also weakly $C$-regular.

\item[(ii)] Suppose now that each $S_{i}$, $i\in I$, is boundedly regular.
Then, by Theorem \ref{t-Rk1} (ii), the sequence $\{T_{j_{k}}^{k}\}_{k=0}^{%
\infty }$ is boundedly $C_{i}$-regular and, consequently, by Theorem \ref%
{t-Rk2} (ii), the sequence $\{T_{k}\}_{k=0}^{\infty }$ is boundedly $C_{i}$%
-regular. Moreover, if we assume that the family $\{C_{i}\mid i\in I\}$ is
boundedly regular, then the sequence $\{T_{k}\}_{k=0}^{\infty }$ is
boundedly $C$-regular.

\item[(iii)] Finally, suppose that each $S_{i}$, $i\in I$, is boundedly
linearly regular and that each subfamily of $\{C_{i}\mid i\in I\}$ is
boundedly linearly regular. Then, by Theorem \ref{t-Rk1} (iii), the sequence 
$\{T_{j_{k}}^{k}\}_{k=0}^{\infty }$ is boundedly linearly $C_{i}$-regular
and, consequently, by Theorem \ref{t-Rk2} (iii), the sequence $%
\{T_{k}\}_{k=0}^{\infty }$ is also boundedly linearly $C_{i}$-regular.
Moreover, the sequence $\{T_{k}\}_{k=0}^{\infty }$ is boundedly linearly $C$%
-regular too.
\end{enumerate}

\item[(b)] (\textit{String averaging sequence}) Let $%
I_{j}^{k}:=(i_{1j}^{k},i_{2j}^{k},...,i_{sj}^{k})\subseteq I$ be an ordered
subset, where $s\geq 1$, $j\in J$ and $k=0,1,2,\ldots $. Let%
\begin{equation}
T_{j}^{k}:=\prod_{i\in I_{j}^{k}}S_{i,\lambda _{ij}^{k}}\text{,}
\label{e-Tjk2}
\end{equation}%
where $I_{j}^{k}\subseteq I$, $\lambda _{ij}^{k}\in (0,1)$, $i\in I_{j}^{k}$%
, $j\in J$, $k=0,1,2,\ldots $, and%
\begin{equation}
T_{k}:=\sum_{j\in J_{k}}\nu _{j}^{k}T_{j}^{k}\text{,}  \label{e-Tk2}
\end{equation}%
where $J_{k}\subseteq J$, $\nu _{j}^{k}\geq \delta >0$, $j\in J_{k}$, $%
\sum_{j\in J_{k}}\nu _{j}^{k}=1$ and $k=0,1,2,\ldots $. The string averaging
methods for solving the convex feasibility problem \cite{RZ16} can be
represented in the form $x^{k+1}=T_{k}x^{k}$ with a sequence of operators $%
T_{k}$ given by (\ref{e-Tk2}), where $T_{j}^{k}$ are defined by (\ref{e-Tjk2}%
). Denote $\rho _{ij}^{k}:=(1-\lambda _{ij}^{k})/\lambda _{ij}^{k}$, $i\in
I_{j}^{k}$, $j\in J$, $k=0,1,2,\ldots $, and $\rho :=\inf_{k\geq
0}\min_{i\in I_{j}^{k},j\in J}\rho _{ij}^{k}$, and suppose that $\underline{%
\lambda }:=\inf_{k\geq 0}\min_{i\in I_{j}^{k},j\in J}\lambda _{ij}^{k}>0$
and $\bar{\lambda}:=\sup_{k\geq 0}\max_{i\in I_{j}^{k},j\in J}\lambda
_{ij}^{k}<1$. Similarly to the situation in (a), $S_{i,\lambda _{ij}^{k}}$
is $\rho _{ij}^{k}$-SQNE and $\rho _{ij}^{k}\geq \rho \geq (1-\bar{\lambda})/%
\bar{\lambda}>0$ for all $i\in I_{j}^{k}$, $j\in J$ and $k\geq 0$. By Fact %
\ref{f-6}(i), $T_{j}^{k}$ is $\rho /m$-SQNE and $\limfunc{Fix}%
T_{j}^{k}=\bigcap_{i\in I_{j}^{k}}C_{i}$. Suppose that $I^{k}:=\bigcup_{j\in
J_{k}}I_{j}^{k}=I$ for all $k=0,1,2,\ldots $. Then $F_{0}:=\bigcap_{k\geq
0}\bigcap_{j\in J_{k}}\limfunc{Fix}T_{j}^{k}=C$. Let $i\in I$ be arbitrary
but fixed, and $i_{k}\in \{1,2,...,s\}$ and $j_{k}\in J$ be such that $%
i_{i_{k},j_{k}}^{k}=i$. Similarly to the situation in (a), by interchanging
Theorem \ref{t-Rk1} with Theorem \ref{t-Rk2}, one can obtain corresponding
statements to (i), (ii) and (iii), respectively.
\end{enumerate}
\end{example}

\begin{remark}
\rm\ %
We now comment on the existing literature, where one can find preservation
of regularity properties under convex combinations and compositions of
operators.

The preservation of weak regularity for a single operator presented in
Corollaries \ref{c-Rk1b} (i) and \ref{c-Rk2b} (i) can be found in \cite[%
Theorem 4.1 and 4.2]{Ceg15a}, respectively. The results concerning bounded
regularity from Corollaries \ref{c-Rk1b} (ii) and \ref{c-Rk2b} (ii) were
shown in \cite[Theorem 4.10 and Theorem 4.11]{CZ14}. Statement (iii) from
the above-mentioned corollaries regarding linear regularity is new, as far
as we know, even in this simple setting.

The preservation of weak regularity for a sequence of operators (Corollaries %
\ref{c-Rk1a} (i) and \ref{c-Rk2a} (i)) was established in \cite[Example 4.5]%
{Ceg15}. These results also follow from \cite[Lemma 3.4]{RZ16}. Preservation
of bounded regularity for a sequence of operators (Corollaries \ref{c-Rk1a}
(ii) and \ref{c-Rk2a} (ii)) was proved in \cite[Lemma 4.10]{Zal14} and \cite[%
Lemma 3.5]{RZ16}. The preservation of linear regularity for a sequence of
operators has not been studied so far.

We would like to emphasize that Theorems \ref{t-Rk1} and \ref{t-Rk2} are
more general than all of the above results.
\end{remark}

\section{Applications}

\label{s-app}

In this section we show how to apply weakly (boundedly, boundedly linearly)
regular sequences of operators to methods for solving convex feasibility and
variational inequality problems.

\subsection{Applications to convex feasibility problems}

\begin{theorem}
\label{th:main} Let $C\subseteq \mathcal{H}$ be nonempty, closed and convex,
and for each $k=0,1,2,\ldots $, let $U_{k}\colon \mathcal{H}\rightarrow 
\mathcal{H}$ be $\rho _{k}$-strongly quasi-nonexpansive with $\rho
:=\inf_{k}\rho _{k}>0$ and $C\subseteq F:=\bigcap_{k=0}^{\infty }\limfunc{Fix%
}U_{k}$. Moreover, let $x^{0}\in \mathcal{H}$ and for each $k=0,1,2,\ldots $%
, let $x^{k+1}:=U_{k}x^{k}$.

\begin{enumerate}
\item[(i)] If $\{U_{k}\}_{k=0}^{\infty }$ is weakly $C$-regular, then $x^{k}$
converges weakly to some $x^{\ast }\in C$.

\item[(ii)] If $\{U_{k}\}_{k=0}^{\infty }$ is boundedly $C$-regular, then
the convergence to $x^{\ast }$ is in norm.

\item[(iii)] If $\{U_{k}\}_{k=0}^{\infty }$ is boundedly linearly $C$%
-regular, then the convergence is $R$-linear, that is, $\Vert x^{k}-x^{\ast
}\Vert \leq 2d(x^{0},C)q^{k}$, $k=0,1,2,\ldots,$ for some $q\in (0,1)$.
\end{enumerate}
\end{theorem}

\begin{proof}
By the definition of an SQNE operator, for any $z\in C$ we have 
\begin{equation}
\Vert x^{k+1}-z\Vert ^{2}=\Vert U_{k}x^{k}-z\Vert ^{2}\leq \Vert
x^{k}-z\Vert ^{2}-\rho _{k}\Vert U_{k}x^{k}-x^{k}\Vert ^{2}  \label{e-sSQNE}
\end{equation}%
and Lemma \ref{f-6a}(ii) yields that 
\begin{equation}
\lim_{k\rightarrow \infty }\Vert U_{k}x^{k}-x^{k}\Vert =0.  \label{e-AR}
\end{equation}

(i) Let $\{U_{k}\}_{k=0}^{\infty }$ be weakly $C$-regular, $x^{\ast }$ be an
arbitrary weak cluster point of $\{x^{k}\}_{k=0}^{\infty }$ and let $%
x^{n_{k}}\rightharpoonup x^{\ast }$. Then $x^{\ast }\in C$. Since $x^{\ast }$
is an arbitrary weak cluster point of $\{x^{k}\}_{k=0}^{\infty }$, Fact \ref%
{f-7}(i) yields the weak convergence of the whole sequence $%
\{x^{k}\}_{k=0}^{\infty }$ to $x^{\ast }$.

(ii) Assume that $\{U_{k}\}_{k=0}^{\infty }$ is boundedly $C$-regular. This,
when combined with (\ref{e-AR}), gives $d(x^{k},C)\rightarrow 0$, which by
Fact \ref{f-7}(ii) implies that the convergence to $x^{\ast }$ is in norm.

(iii) The bounded linear $C$-regularity of $\{U_{k}\}_{k=0}^{\infty }$ and
the boundedness of $\{x^{k}\}_{k=0}^{\infty }$ imply that there is $\delta
>0 $ such that $\Vert U_{k}x^{k}-x^{k}\Vert \geq \delta ^{-1}d(x^{k},C)$
holds for each $k=0,1,2,\ldots $. Consequently, by substituting $z=P_{C}x^{k}
$ into (\ref{e-sSQNE}) and by the inequality $d(x^{k+1},C)\leq \Vert
x^{k+1}-P_{C}x^{k}\Vert $, we arrive at 
\begin{equation}
\rho \delta ^{-2}d^{2}(x^{k},C)\leq d^{2}(x^{k},C)-d^{2}(x^{k+1},C)\text{,}
\end{equation}%
$k=0,1,2,\ldots $. This, when combined with Fact \ref{f-7}(iii), leads to 
\begin{equation}
\Vert x^{k}-x^{\ast }\Vert \leq 2d(x^{0},C)\left( \sqrt{1-\rho /\delta ^{2}}%
\right) ^{k}\text{,}
\end{equation}%
$k=0,1,2,\ldots $, which completes the proof.
\end{proof}

\bigskip

The assumption that $\{U_{k}\}_{k=0}^{\infty }$ is weakly regular (boundedly
regular, boundedly linearly regular) is quite strong. Indeed, by Proposition %
\ref{l-sub-BR}, we have $\bigcap_{l=0}^{\infty}\limfunc{Fix}U_{m_{l}} =
\bigcap_{k=0}^{\infty }\limfunc{Fix}U_{k}$ for any sequence $%
\{m_{k}\}_{k=0}^{\infty }\subseteq \{k\}_{k=0}^\infty$. Consequently, we
cannot directly apply Theorem \ref{th:main} in the case of (almost) cyclic
or intermittent control. This is due to the fact that for some subsequence $%
\{m_k\}_{k=0}^\infty$, one could have $\bigcap_k\limfunc{Fix}U_{m_k}\neq
\bigcap_k \limfunc{Fix}U_k$. Nevertheless, Theorem \ref{th:main} can still
be indirectly applied to the above-mentioned controls as we now show.

\begin{theorem}
\label{th:main2} Let $C_{i}\subseteq \mathcal{H}$ be closed and convex, $%
i\in I:=\{1,\ldots ,m\}$, such that $C:=\bigcap_{i\in I}C_{i}\neq \emptyset $%
. For each $k=0,1,2,\ldots,$ let $U_{k}\colon \mathcal{H}\rightarrow 
\mathcal{H}$ be $\rho_{k}$-strongly quasi-nonexpansive with $C\subseteq
F:=\bigcap_{k=0}^{\infty }\limfunc{Fix}U_{k}$, and let $\rho :=\inf_{k}\rho
_{k}>0$. Moreover, let $x^{k}$ be generated by $x^{k+1}:=U_{k}x^{k}$, $%
k=0,1,2,\ldots$, where $x^{0}\in \mathcal{H}$ is arbitrary. In addition, let 
$\{n_{k}^{i}\}_{k=0}^{\infty }$, $i\in I$, be increasing sequences of
nonnegative integers with bounded growth, that is, $0<n_{k+1}^{i}-n_{k}^{i}%
\leq s$, $k=0,1,2,\ldots$, for some $s>0$. If for each $i\in I$, the
subsequence $\{U_{n_{k}^{i}}\}_{k=0}^{\infty }$ is:

\begin{enumerate}
\item[$\mathrm{(i)}$] weakly $C_{i}$-regular, then $\{x^{k}\}_{k=0}^{\infty
} $ converges weakly to some $x^{\ast }\in C$.

\item[$\mathrm{(ii)}$] boundedly $C_{i}$-regular and the family $\{C_{i}\mid
i\in I\}$ is boundedly regular, then the convergence to $x^{\ast }$ is in
norm.

\item[$\mathrm{(iii)}$] boundedly linearly $C_{i}$-regular and the family $%
\{C_{i}\mid i\in I\}$ is boundedly linearly regular, then the convergence is 
$R$-linear, that is, $\Vert x^{k}-x^{\ast }\Vert \leq 2d(x^{0},C)q^{k}$, $%
k=0,1,2,\ldots$, for some $q\in (0,1)$.
\end{enumerate}
\end{theorem}

\begin{proof}
Let 
\begin{equation}
T_{k}:=U_{k+s-1}U_{k+s-2}...U_{k}\text{,}  \label{proof:th:main2:Tk}
\end{equation}%
$k\geq 0$, and define a sequence $\{y^{k}\}_{k=0}^{\infty }$ by 
\begin{equation}
y^{0}:=x^{0};\quad y^{k+1}:=T_{k}y^{k}.  \label{proof:th:main2:yk}
\end{equation}%
Obviously, we have $y^{k}=x^{ks}$, $k=0,1,2,\ldots$..

(i) By Theorem \ref{t-Rk2}(i), the sequence $\{T_{k}\}_{k=0}^{\infty }$ is
weakly $C$-regular. Now Theorem \ref{th:main}(i) yields the weak convergence
of $y^{k}$ to a point $x^{\ast }\in C$. Moreover, by the assumption, $U_{k}$
is $\rho _{k}$-SQNE, and $\rho >0$ yields that $\lim_{k}\Vert
x^{k+1}-x^{k}\Vert =0$ (see Lemma \ref{f-6a}(ii)). In view of Lemma \ref%
{th:Fejer2}(i), these facts yield the weak convergence of $x^{k}$ to $%
x^{\ast }\in C$.

(ii) By Theorem \ref{t-Rk2}(ii), the sequence $\{T_{k}\}_{k=0}^{\infty }$ is 
$C$-regular. Now Theorem \ref{th:main}(ii) yields the convergence in norm of 
$y^{k}$ to $x^{\ast }\in C$. In view of Lemma \ref{th:Fejer2}(ii), this
yields the convergence in norm of $x^{k}$ to $x^{\ast }\in C$.

(iii) By Theorem \ref{t-Rk2}(iii), the sequence $\{T_{k}\}_{k=0}^{\infty }$
is linearly $C$-regular. Now Theorem \ref{th:main}(iii) yields the $R$%
-linear convergence of $y^{k}$ to $x^{\ast }\in C$. In view of Lemma \ref%
{th:Fejer2}(iii), this yields the $R$-linear convergence of $x^{k}$ to $%
x^{\ast }\in C$.
\end{proof}

\subsection{Applications to variational inequality problems}

Let $G\colon\mathcal{H}\rightarrow\mathcal{H}$ be monotone and let $%
C\subseteq\mathcal{H}$ be nonempty, closed and convex. We recall that the 
\textit{variational inequality problem} governed by $G$ and $C$, which we
denote by VI($G$,$C$), is to 
\begin{equation}
\text{find }\bar{u}\in C\text{ with }\langle G\bar{u},u-\bar{u}\rangle \geq 0%
\text{ for all }u\in C.  \label{e-VIP1}
\end{equation}%
It is well known that VI($G$,$C$) has a unique solution if, for example, $G$
is $\kappa $-Lipschitz continuous and $\eta$-strongly monotone, where $%
\kappa, \eta >0$ \cite[Theorem 46.C]{Zei85}. In this section we show how one
can apply the results of Section \ref{s-4} to an iterative scheme for
solving VI($G$, $C$). We begin with recalling some known results.

\begin{theorem}
\label{t-uk-WR} Let $G\colon\mathcal{H}\rightarrow\mathcal{H}$ be $\kappa$%
-Lipschitz continuous and $\eta$-strongly monotone, where $\kappa,\eta>0$,
and let $C\subseteq\mathcal{H}$ be nonempty, closed and convex. Moreover,
for each $k=0,1,2,\ldots,$ let $U_k\colon\mathcal{H}\rightarrow\mathcal{H}$
be $\rho_k$-strongly quasi-nonexpansive such that $C\subseteq \limfunc{Fix}
U_k$ and $\lambda_k\in[0,\infty)$. Consider the following method: 
\begin{equation}
u^0\in\mathcal{H}; \quad u^{k+1}=U_{k}u^{k}-\lambda _{k}GU_{k}u^{k}.
\label{e-uk}
\end{equation}%
Assume that $\rho:=\inf_k\rho_k>0$, $\lim_k\lambda_k =0$ and $%
\sum_k\lambda_k =\infty$. If $\{U_{k}\}_{k=0}^{\infty }$ is weakly $C$%
-regular (in particular, boundedly $C$-regular), then $u^{k}$ converges
strongly to the unique solution of VI($G$, $C$).
\end{theorem}

\begin{proof}
See \cite[Theorem 4.8]{Ceg15}. For related results see also \cite[Theorem 4.3%
]{AK14} and \cite[Th. 2.4]{Hir06}. The part regarding a boundedly regular
sequence of operators follows from Corollary \ref{c-reg}(ii) in view of
which a boundedly $C$-regular sequence is also weakly $C$-regular.
\end{proof}

\bigskip

As we mentioned in the previous subsection, the assumption that $%
\{U_{k}\}_{k=0}^{\infty }$ is weakly regular is quite strong. In particular,
this assumption excludes (almost) cyclic and intermittent controls. In the
next result we show that one can still establish norm convergence for method %
\eqref{e-uk} in the case of the above-mentioned controls, but at the cost of
imposing bounded regularity of both families of operators and sets.

\begin{theorem}
\label{t-uku} Let $G\colon\mathcal{H}\rightarrow\mathcal{H}$ be $\kappa$%
-Lipschitz continuous and $\eta$-strongly monotone, where $\kappa,\eta>0$,
and let $C:=\bigcap_{i\in I}C_i\subseteq\mathcal{H}$ be nonempty, where for
each $i\in I:=\{1\ldots,m\}$, $C_i$ is closed and convex. Moreover, for each 
$k=0,1,2,\ldots,$ let $U_k\colon\mathcal{H}\rightarrow\mathcal{H}$ be $%
\rho_k $-strongly quasi-nonexpansive such that $C\subseteq \limfunc{Fix} U_k$%
, $\lambda_k\in[0,\infty)$ and consider the following method: 
\begin{equation}
u^0\in\mathcal{H}; \quad u^{k+1}=U_{k}u^{k}-\lambda _{k}GU_{k}u^{k}.
\label{e-uk2}
\end{equation}%
Assume that $\rho:=\inf_k\rho_k>0$, $\lim_k\lambda_k =0$ and $%
\sum_k\lambda_k =\infty$. Moreover, assume that there is $s\geq 0$ such that
for any $i\in I$ and $k=0,1,2,\ldots$, there is $l_{k}\in
\{k,k+1,...,k+s-1\} $ such that the subsequence $\{U_{l_{k}}\}_{k=0}^{\infty
}$ is $C_{i}$-regular and that the family $\{C_{i}\mid i\in I\}$ is
boundedly regular. Then $u^{k}$ converges strongly to the unique solution of
VI($G$, $C$).
\end{theorem}

\begin{proof}
By \cite[Theorem 2.17]{GRZ17}, it suffices to show that the implication 
\begin{equation}
\lim_{k}\sum_{l=n_{k}}^{n_{k}+s-1}\Vert U_{l}u^{l}-u^{l}\Vert=0
\Longrightarrow \lim_{k}d(u^{n_{k}},C)=0
\end{equation}
holds true for any arbitrary subsequence $\{n_{k}\}_{k=0}^{\infty }\subseteq
\{k\}_{k=0}^{\infty }$. To this end, choose $\{n_{k}\}_{k=0}^{\infty
}\subseteq \{k\}_{k=0}^{\infty }$ and assume that 
\begin{equation}  \label{t-uku:proof:1}
\lim_{k}\sum_{l=n_{k}}^{n_{k}+s-1}\Vert U_{l}u^{l}-u^{l}\Vert=0.
\end{equation}
Let $i\in I$ be arbitrary. By assumption, for each $k=0,1,2,\ldots,$ there
is $l_{k}\in \{n_{k},n_{k}+1,...,n_{k}+s-1\}$ such that $\{U_{l_{k}}%
\}_{k=0}^{\infty }$ is $C_{i}$-regular. So, by the boundedness of $%
\{u^k\}_{k=0}^\infty$ (see \cite[Lemma 9]{CZ13}) and \eqref{t-uku:proof:1},
we have 
\begin{equation}
\lim_{k}d(u^{l_{k}},C_{i})=0\text{.}  \label{e-lim-dujk}
\end{equation}%
Observe that the boundedness of $u^{k}$ and $\lim_{k}\lambda _{k}=0$ lead to 
\begin{equation}
\lim_{k}\sum_{l=k}^{k+s-1}\lambda _{l}\Vert GU_{l}u^{l}\Vert =0\text{.}
\label{e-sum-GUk}
\end{equation}%
Moreover, the triangle inequality, \eqref{e-uk2}, \eqref{t-uku:proof:1} and %
\eqref{e-sum-GUk} imply that

\begin{align}
\Vert u^{n_{k}}-u^{l_{k}}\Vert & \leq \sum_{l=n_{k}}^{l_{k}}\Vert
u^{l+1}-u^{l}\Vert \leq \sum_{l=n_{k}}^{n_{k}+s-1}\Vert u^{l+1}-u^{l}\Vert 
\notag \\
& \leq \sum_{l=n_{k}}^{n_{k}+s-1}\Vert U_{l}u^{l}-u^{l}\Vert
+\sum_{l=n_{k}}^{n_{k}+s-1}\lambda _{l}\Vert GU_{l}u^{l}\Vert \rightarrow_k
0.
\end{align}
This, the definition of the metric projection, the triangle inequality and (%
\ref{e-lim-dujk}) yield

\begin{align}
d(u^{n_{k}},C_{i}) & =\Vert u^{n_{k}}-P_{C_{i}}u^{n_{k}}\Vert \leq \Vert
u^{n_{k}}-P_{C_{i}}u^{l_{k}}\Vert  \notag \\
& \leq \Vert u^{n_{k}}-u^{l_{k}}\Vert +\Vert
u^{l_{k}}-P_{C_{i}}u^{l_{k}}\Vert \rightarrow_k 0.
\end{align}
Since $i\in I$ is arbitrary and the family $\{C_{i}\mid i\in I\}$ is
boundedly regular, $\lim_{k}d(u^{n_k},C)=0$, which completes the proof.
\end{proof}

\begin{remark}
\rm\ %
\cite[Theorem 2.17]{GRZ17}, which we have used in order to prove Theorem \ref%
{t-uku}, appeared for the first time in \cite[Theorem 3.16]{Zal14}. Since
this result was presented in Polish, we refer here to a paper which has been
published in English. Related results can be found, for example, in \cite[%
Theorem 12]{CZ13} or \cite[Theorem 4.13]{Ceg15}.
\end{remark}

\noindent \textbf{Funding.} The research of the second author was supported
in part by the Israel Science Foundation (Grants no. 389/12 and 820/17), the
Fund for the Promotion of Research at the Technion and by the Technion
General Research Fund.

\textbf{Acknowledgments.} We are grateful to an anonymous referee for
his/her comments and remarks which helped us to improve our manuscript.

\end{document}